\documentclass[a4paper,12pt,reqno]{amsart}
\usepackage{enumerate}
\usepackage{mathrsfs,stmaryrd}
\usepackage{graphicx}
\usepackage{amsthm,amsmath,latexsym,amscd}
\usepackage{txfonts,amssymb}

\makeatletter

\@addtoreset{equation}{section}

\@addtoreset{figure}{section}
\makeatother

\theoremstyle{plain}
\newtheorem{Theorem}{Theorem}[section]
\newtheorem{theorem0}{Thorem}
\newtheorem{Proposition}[Theorem]{Proposition}
\newtheorem{Lemma}[Theorem]{Lemma}
\newtheorem{Corollary}[Theorem]{Corollary}
\theoremstyle{definition} 
\newtheorem{Algorithm}[Theorem]{Algorithm}
\newtheorem{Definition}[Theorem]{Definition}
\newtheorem{Proposition-Definition}[Theorem]{Proposition-Definition}
\newtheorem{Remark}[Theorem]{Remark}
\newtheorem{Example}[Theorem]{Example}

\newtheorem{Observation}[Theorem]{Observation}

\def\T{{\mathcal{T}}}

\def\Z{{\varmathbb{Z}}}
\def\N{{\varmathbb{N}}}
\def\Q{{\varmathbb{Q}}} 
\def\R{{\varmathbb{R}}} 

\def\F{{\varmathbb{F}}}
\def\k{{\varmathbb{K}}}

\def\x{{\boldsymbol{x}}}

\def\a{{\boldsymbol{a}}}
\def\b{{\boldsymbol{b}}}
\def\c{{\boldsymbol{c}}}

\def\pR{\widetilde{\R}_+}

\def\pN{\widetilde{\N}_+}

\def\w{{\boldsymbol{w}}}
\def\bw{{\boldsymbol{\omega}}}
\def\bsxi{{\boldsymbol{\xi}}}
\def\bseta{{\boldsymbol{\eta}}}

\def\M{{\mathcal{M}}}

\def\p{{\mathfrak{p}}}
\def\P{{\mathfrak{P}}}
\def\q{{\mathfrak{q}}}

\def\Coker{\mathop{\mathrm{Coker}}\nolimits}

\def\Semi{\mathop{\mathrm{Semi}}\nolimits}

\def\Prim{\mathop{\mathrm{Prim}}\nolimits}

\def\ch{\mathop{\mathrm{char}}\nolimits}

\def\Min{\mathop{\mathrm{Min}}\nolimits}
\def\Ass{\mathop{\mathrm{Ass}}\nolimits}
\def\LT{\mathop{\mathrm{LT}}\nolimits}

\def\Ker{\mathop{\mathrm{Ker}}\nolimits}
\def\deg{\mathop{\mathrm{deg}}\nolimits}
\def\int{\mathop{\mathrm{int}}\nolimits}

\def\initial{\mathop{\mathrm{in}}\nolimits}

\def\ord{\mathop{\mathrm{ord}}\nolimits}

\def\notdivide{\setbox1=\hbox{$|$\llap{\hbox{/}\kern-0.75pt}}\mathrel{\box1}}
\providecommand{\abs}[1]{\lvert#1\rvert}

\def\subst{\mbox{$\ \longleftarrow$}}                        

\newcommand{\pscif}[1]{{\mbox{\textit{\underline{\textbf{if}}}}}\ \big(#1\big) \mbox{\textit{\underline{\textbf{then}}}}}
\newcommand{\pscendif}{{\mbox{\textit{\underline{\textbf{end if}}}}}\ \ }
\newcommand{\pscforall}[1]{{\mbox{\textit{\underline{\textbf{for all}}}}}\ \big(#1\big) \mbox{\textit{\underline{\textbf{do}}}}}
\newcommand{\pscendforall}{{\mbox{\textit{\underline{\textbf{end for}}}}}\ \ }

\newcommand{\pscwhile}[1]{{\mbox{\textit{\underline{\textbf{while}}}}}\ \big(#1\big) \mbox{\textit{\underline{\textbf{do}}}}}
\newcommand{\pscendwhile}{{\mbox{\textit{\underline{\textbf{end while}}}}}\ \ }
\newcommand{\pscreturn}{{\mbox{\textit{\underline{\textbf{return}}}}}\ \ }

\def\pscinputstr{Input:}
\def\pscoutputstr{Output:}
\def\pscinputname{\makebox[\pscioboxwidth][l]{\textbf{\pscinputstr} }}
\def\pscoutputname{\makebox[\pscioboxwidth][l]{\textbf{\pscoutputstr} }}

\makeatletter
\newcounter{pseudocode}
\def\pseudocodename{\textbf{Algorithm}}%
\def\thepseudocode{\thesection.\arabic{pseudocode}}%
\def\fps@pseudocode{!htb}%
\def\ftype@pseudocode{4}%
\def\ext@pseudocode{loc}%
\def\fnum@pseudocode{\pseudocodename~\thepseudocode}%
{
	\@dblfloat{pseudocode}%
	\pscindent=0zw%
	\caption{\textbf{#1}}\label{#2}%
	\rule[1pt]{\linewidth}{1pt}%
	\newcount\pscinum\pscinum=0%
	\newcount\psconum\psconum=0%
	\begin{list}{\theenumi}{
		\topsep.0\baselineskip\partopsep0pt%
		\leftmargin1zw%
		\labelwidth2zw\labelsep0.5zw%
		\itemindent0pt\rightmargin0pt
		\listparindent0pt\itemsep0pt\parsep0pt%
		\usecounter{enumi}%
		\def\makelabel##1{\ifx##1n{\ }\else{\makebox[2zw][r]{##1:}}\fi}%
		\newcommand{\pscinput}{\item[n]\ifnum\pscinum=0{\pscinputname}\else{\hspace{\pscioboxwidth}}\fi\advance \pscinum by 1}%
		\newcommand{\pscoutput}{\item[n]\ifnum\psconum=0{\pscoutputname}\else{\hspace{\pscioboxwidth}}\fi\advance \psconum by 1}%
	}
}
{
	\end{list}
	\rule[1pt]{\linewidth}{1pt}
	\end@dblfloat
}
\makeatother

\begin{document}
\title{Irreducibility criterion for algebroid curves}
\author{Takafumi Shibuta
}
\address{Department of Mathematics, Rikkyo University, 
Nishi-Ikebukuro, Tokyo 171-8501, Japan}
\email{shibuta@rikkyo.ac.jp}
\date{}
\baselineskip 15pt
\footskip = 32pt
\begin{abstract}
The purpose of this paper is to give an algorithm for deciding the irreducibility of reduced algebroid curves over an algebraically closed field of arbitrary characteristic. 
To do this, we introduce a new notion of local tropical variety which is a straightforward extension of tropism introduced by Maurer, 
and then give irreducibility criterion for algebroid curves in terms local tropical varieties. 
We also give an algorithm for computing the value-semigroups of irreducible algebroid curves. 
Combining the irreducibility criterion and the algorithm for computing the value-semigroups, we obtain 
an algorithm for deciding the irreducibility of algebroid curves. 
\end{abstract}
\maketitle
\tableofcontents
\section{Introduction}
Throughout this paper, $k$ is a perfect field of arbitrary characteristic, and $K$ is the algebraic closure of $k$. 
An {\it algebroid curve} over $k$ is a Noetherian local ring $A$ such that: 
\begin{enumerate}
\item $A$ is complete. 
\item $A$ is unmixed and of Krull dimension one (i.e. $\dim \p=1$ for all $\p\in \Ass A$). 
\item $A$ has a coefficient field $k$. 
\end{enumerate}
If $A$ is domain, we say that $A$ is {\it irreducible}, 
and if $A$ has no nilpotent element, we say that $A$ is {\it reduced}. 
By Cohen's theorem, $A$ is isomorphic to $k\llbracket \x \rrbracket/I$ for some $k\llbracket \x\rrbracket=k\llbracket x_1,\dots,x_r\rrbracket$ and $I$. 
The purpose of this paper is to give an algorithm for deciding the irreducibility of a reduced algebroid curve $K\llbracket \x\rrbracket/I$ 
over the algebraically closed field $K$. 
As Teitelbaum pointed out, the question of determining whether or not a given algebroid curve is reduced is uncomputable in general (\cite{Teitelbaum} Lemma 1). 
When we consider the implementation, we assume that $k$ is a computable perfect field (e.g. $\Q$ or $\F_p$), 
and an ideal $I\subset K\llbracket \x\rrbracket$ is generated by polynomials over $k$. 

For plane algebroid curves over algebraic closed field of characteristic zero, Abhyankar \cite{Ab} has given an irreducibility criterion for bivariate power series, 
and Kuo \cite{Kuo} has presented a simple algorithm for deciding irreducibility of bivariate power series. 

In Section 3, we will give a criterion for an algebroid curve $k\llbracket \x \rrbracket/I$ to be irreducible. 
To do this, we introduce a new notion of {\em local tropical variety} in Section 2. 
Local tropical variety is a straightforward extension of {\em tropism} introduced by Maurer \cite{Maurer}, 
and an analogue of tropical variety in formal power series rings. 
See \cite{BG}, \cite{EM}, \cite{SS}, \cite{Sturmfels2} for tropical varieties. 
We denote by $\N=\{0,1,2,3,\dots\}$, $\N_+$, and $\R_+$, the set of non-negative integers, positive integers, and positive real numbers, respectively. 
We set $\pN=\N_+\cup \{\infty\}$ and $\pR=\R_+\cup \{\infty\}$. 
For $\w=(w_1,\dots,w_r)\in \pN^r$, we set $\gcd(\w)=\gcd(w_i\mid w_i\neq \infty)$. 
The local tropical varieties $\T_{\rm loc}(I)$ of $I\subset k\llbracket \x\rrbracket$ is defined as the set of weight vectors $\w\in \pN^r$ 
such that the initial ideal of $I$ with respect to $\w$ contains no monomial (Definition \ref{ltv}). 
We say that $\w\in \pN^r$ is a tropism of $I$ if $\in \T_{\rm loc}(I)$ and $\gcd(\w)=1$. 
We set $\bw(I)=(\int(x_1;I),\dots,\int(x_r;I))\in \pN^r$ for $I\subset k\llbracket \x\rrbracket$ 
where $\int(x_i; I)=\dim_k k\llbracket \x\rrbracket/ \langle x_i, I\rangle$ (Definition \ref{def of wI}). 
The next theorem is the irreducibility criterion for algebroid curves in terms of local tropical varieties. 
\begin{theorem0}[Theorem \ref{primeness}]\label{intro primeness}
Let $I\subset k\llbracket \x\rrbracket$ be an unmixed ideal of dimension one. 
Then the following hold. 
\begin{enumerate}
\item 
If $\T_{\rm loc}(I)\neq \pR\cdot \bw(I)$ (e.g. $\initial_{\bw(I)}(I)$ contains monomials, or $I$ has at least two tropisms) 
then $I$ is not prime. 
\item 
If $\bw(I)$ is a tropism of $I$, then $I$ is prime. 
\item 
Let $\{\w_1,\dots,\w_l\}\subset \pN^r$ be the set of the tropisms of $I$. 
If $\bw(I)\neq(\infty,\dots,\infty)$ and $\bw(I)=\w_1+\dots+\w_l$, 
then there exist prime ideals $\p_j\subset k\llbracket\x\rrbracket$ for $1\le j\le l$ 
such that $\bw(\p_j)=\w_j$ and $I=\bigcap_{j=1}^l \p_j$. 
\end{enumerate}
\end{theorem0}

In Section 5, 
we will present an algorithm for deciding the irreducibility of reduced algebroid curves (Algorithm \ref{main algorithm}). 
For a one-dimensional radical ideal $I\subset K\llbracket x_1,\dots,x_r\rrbracket$, 
Algorithm \ref{main algorithm} compute an ideal $J\subset K\llbracket x_1,\dots,x_{r'}\rrbracket$, $r'\ge r$, 
satisfying the following; 
\begin{enumerate}
\item $K\llbracket x_1,\dots,x_r\rrbracket/I\cong K\llbracket x_1,\dots,x_{r'}\rrbracket/J$, 
\item if $I$ is prime, then $\bw(J)$ is a tropism of $J$, 
\item if $I$ is not prime, then $J$ has at least two tropisms. 
\end{enumerate}
To do this, we introduce a new notion of {\em local SAGBI basis} which is a variant of SAGBI basis (or canonical subalgebra basis) in local rings in Section 4, 
and then give an algorithm for computing the value-semigroup of an irreducible algebroid curve (Algorithm \ref{Semi(p)}). 
We obtain Algorithm \ref{main algorithm} by combining Algorithm \ref{Semi(p)} and the above irreducibility criterion. 
On smaller examples, Algorithm \ref{main algorithm} can be carried out by hand, and  
Algorithm \ref{main algorithm} is implementable for computer algebra systems. 
\section{Local tropical varieties}
\subsection{Definition of local tropical varieties}
In this section, we introduce a notion of local tropical varieties, and prove its basic properties. 

As in Introduction, $k$ is a perfect field of arbitrary characteristic, and $K$ is the algebraic closure of $k$. 
For $\x=(x_1,\dots,x_r)$ and $\a=(a_1,\dots,a_r)$, we use multi-index notation $\x^\a=x_1^{a_1}\dots x_r^{a_r}$. 
Let $k\llbracket\x\rrbracket=k\llbracket x_1,\dots,x_r\rrbracket$ be a formal power series ring over the algebraically closed field $k$. 
Recall that $\pN=\N_+\cup \{\infty\}$, and $\pR=\R_+\cup \{\infty\}$. 
\begin{Definition}
Let $\w=(w_1,\dots,w_r)\in \pR^r$. 
For $f=\sum_{\a\in\N^r} c_\a \x^\a\in {k}\llbracket\x\rrbracket$, $c_\a\in {k}$, we define the {\it order} of $f$ with respect to $\w$ as 
\[
\ord_\w(f)=\min\{\w\cdot\a\mid c_\a\neq 0 \}\in \pR\cup \{0\}, 
\]
where $\w\cdot\a=\sum w_ia_i\in \pR$. 
We set $\ord_\w(0)=\infty$. 
If $\ord_\w(f)<\infty$, we define the {\it initial form} of $f$ as 
\[
\initial_\w(f)=\sum_{\w\cdot\a=\ord_\w(f)}c_\a \x^\a\in {k}[x_i\mid w_i\neq \infty].
\]
If $\ord_\w(f)=\infty$, we set $\initial_\w(f)=0$. 

For an ideal $I\subset {k}\llbracket\x\rrbracket$, we call 
$\initial_\w(I)=\langle \initial_\w(f)\mid f\in I\rangle\subset {k}[x_i\mid w_i\neq \infty]$ the {\it initial ideal} of $I$ with respect to $\w$ . 
\end{Definition}
\begin{Example}
Let $\w=(1,2,\infty)$. 
Then $\ord_\w(x^2+y+xy+z)=\min\{2,2,3,\infty\}=2$, $\initial_\w(x^2+y+xy+z)=x^2+y$, 
and $\ord_\w(z)=\infty$, $\initial_\w(z)=0$. 
\end{Example}
Now, we define local tropical varieties. 
Our definition is different form Touda's definition (\cite{Touda} Definition 4.7). 
For $\w=(w_1,\dots,w_r)\in \pN^r$, $\w\neq(\infty,\dots,\infty)$, we set $\gcd(\w)=\gcd\{w_i\mid w_i\neq\infty\}$. 
We say that $\w$ is {\em primitive} if $\gcd (\w) = 1$. 
\begin{Definition}\label{ltv}
Let $I\subset k\llbracket\x\rrbracket$ be an ideal. We call 
\[
\T_{\rm loc}(I)=\{\w\in\pR^r \mid \initial_\w(I)\mbox{ contains no monomial }\}. 
\]
the {\it local tropical variety} of $I$. 
We say that an element $\w\in \T_{\rm loc}(I)\cap \pN^r$ is a {\it tropism} of $I$ if $\w$ is primitive. 
\end{Definition}
The initial ideal $\initial_\w(I)$ is homogeneous with respect to $\w$, 
and thus $\initial_\w(I)$ contains no monomial if and only if $\initial_\w(f)$ is not monomial for any $f\in I$. 
The topological closure of $\T_{\rm loc}(I)\cap \R_+^r$ in $\R^r$ is a rational polyhedral complex (\cite{BT} Theorem 2.0.4). 
If $I$ is generated by polynomials, then $\T_{\rm loc}(I)$ is computable \cite{BJSS}. 

Similarly to the usual varieties, local tropical varieties satisfy the following properties. 
\begin{Lemma}\label{intersection of ideals}
Let $I,J, I_1,\dots,I_l\subset k\llbracket\x\rrbracket$ be ideals, and $V,W\subset \M^{\oplus r}$. Then the following hold. 
\begin{enumerate}
\item If $I\subset J$, then $\T_{\rm loc}(I)\supset\T_{\rm loc}(J)$. 
\item $\sqrt{\initial_\w(I)}=\sqrt{\initial_\w(\sqrt{I})}$ for any $\w\in \pR^r$. 
In particular, $\T_{\rm loc}(I)=\T_{\rm loc}(\sqrt{I})$. 
\item $\T_{\rm loc}(\bigcap_{i=1}^l I_i)=\bigcup_{i=1}^l\T_{\rm loc}(I_i)$. 
\end{enumerate}
\end{Lemma}
\begin{proof}
(1) is trivial. 

(2): Since $I\subset\sqrt{I}$, $\sqrt{\initial_\w(I)}\subset\sqrt{\initial_\w(\sqrt{I})}$. 
To prove the converse, let $f\in {\initial_\w(\sqrt{I})}$. 
Then there exists $g\in\sqrt{I}$ such that $\initial_\w(g)=f$. 
Since $g^{n}\in I$ for some $n\in \N$, $f^{n}=\initial_\w(g^{n})\in\initial_\w(I)$. 
Thus $f\in \sqrt{\initial_\w(I)}$. 
This implies ${\initial_\w(\sqrt{I})}\subset \sqrt{\initial_\w(I)}$ and thus $\sqrt{\initial_\w(\sqrt{I})}\subset \sqrt{\initial_\w(I)}$. 

(3): Since $\bigcap I_i\subset I_i$ for all $i$, 
we have $\T_{\rm loc}(\bigcap I_i)\supset\bigcup_{i}\T_{\rm loc}(I_i)$. 
Let $\w\not\in \bigcup_{i}\T_{\rm loc}(I_i)$. 
For each $i$, there exists $f_i\in I_i$ such that $\initial_\w(f_i)$ is a monomial. 
Since $\prod_{i}f_i\in\prod I_i\subset\bigcap I_i$, we conclude that 
$\initial_\w(\bigcap I_i)$ contains a monomial 
$\initial_\w(\prod_{i}f_i)=\prod_{i}\initial_\w(f_i)$. 
\end{proof}
By Lemma \ref{intersection of ideals}, 
$\T_{\rm loc}(I)=\bigcup_{\P \in \Min I} \T_{\rm loc}(\P)$ where $\Min I$ is the set of the minimal associated prime ideals of $I$. 
\subsection{Fundamental theorem of local tropical geometry}
We will prove an analogue of the the fundamental theorem of tropical geometry 
(\cite{SS} Theorem 2.1, \cite{Draisma} Theorem 4.2, \cite{JMM} Theorem 2.13). 
We use the theory of affinoid algebras similarly to the proof of the fundamental theorem of tropical geometry in \cite{Draisma}. 

We set 
\[
\k=K((t^\R))=\Bigl\{ \sum_{a\in \R}c_at^a~~\Big|~~ c_a\in K \mbox{ and } \{a : c_a\neq 0\} \mbox{ is well-ordered}\Bigr\}. 
\]
Then $\k$ is a field in the natural way. 
We define the valuation $\ord_t$ on $\k$ by $\ord(\sum_{a\in \R}c_at^a)=\min\{a\mid c_a\neq 0\}$. 
Fix $e\in \R_+$. For $\xi\in \k$, we define $\abs{\xi}:=e^{-\ord_t(\xi)}$, and call it the {\em non-archimedean absolute value} of $\xi$. 
Then $\k$ is complete with respect to the metric induced by $\abs{~}$. 
We denote by  ${\k^\circ }:=\{\xi \in \k\mid \ord_t(\xi)\ge 0\}=\{\xi \in \k\mid \abs{\xi}\le 1\}$, 
and $\k^{\circ\circ}:=\{\xi \in \k\mid \ord_t(\xi)> 0\}=\{\xi \in \k\mid \abs{\xi}< 1\}$, 
the valuation ring of $\ord_t$, and its unique maximal ideal, respectively. 
For simplicity of notation, we write $\M=\k^{\circ\circ}$ in this paper. 
Then the residue field  and ${\k^\circ}/\M$ is isomorphic to $K$. 
Since $K$ is algebraically closed and $\abs{\k^\times}:=\{\abs{\xi}\mid \xi\in \k^\times\}=e^\R=\R_+$, 
$\k$ is also algebraically closed (see \cite{FM} Section 1.1. The field $\k$ coincides with the field $k((\Gamma))$ in loc. cit., Example 1.1.3 with $\Gamma=e^{\R}$). 

For $\bsxi=(\xi_1,\dots,\xi_r)\in \M^{\oplus r}$, 
$K\llbracket \bsxi \rrbracket \subset \k^\circ$ is well-defined, 
and the substitution $f(\bsxi)\in \k^\circ$ for $f(\x)\in {K}\llbracket \x\rrbracket$ does make sense. 
Hence we can define the following. 
\begin{Definition}\label{def of VM}
Let $I\subset K\llbracket\x\rrbracket$ be a subset (e.g. an ideal of $K\llbracket\x\rrbracket$, or $k\llbracket\x\rrbracket$). We set 
\[
V_\M(I)=\{\bsxi=(\xi_1,\dots,\xi_r)\in \M^{\oplus r} \mid f(\bsxi)=0 \mbox{ for all } f\in I\}. 
\]
\end{Definition}
We denote the {\em Tate algebra} by 
\[
T_r:=\k\langle \x\rangle=\biggl\{ \sum_{\a\in \N^r}c_\a \x^\a \in \k\llbracket\x\rrbracket ~~\Big|~~ c_\a\in \k, \abs{c_\a}\to 0 \mbox{ as } \abs{\a}\to \infty \biggr\}
\]
where $\abs{\a}=a_1+\dots+a_r$ for $\a=(a_1,\dots,a_r)$. 
We extend $\ord_t$ to a function on $T_r$ by setting $\ord_t(\sum_{\a\in \N^r}c_\a \x^\a)=\min\{\ord_t(c_\a)\mid c_\a\neq 0\}$. 
We set 
$T_r^\circ:=\{f\in T_r\mid \ord_t(f)\ge 0\}
=\{ \sum_{\a\in \N^r}c_\a \x^\a \mid c_\a\in {\k^\circ}, \abs{c_\a}\to 0 \mbox{ as } \abs{\a}\to \infty \}$, 
and $T_r^{\circ\circ}:=\{f\in T_r\mid \ord_t(f)> 0\}=\M T_r^\circ$. 
Then $\overline{T}_r:=T_r^\circ/T_r^{\circ\circ}$ is canonically isomorphic to $K[\x]$. 

The next theorem is an analogue of the the fundamental theorem of tropical geometry. 
\begin{Theorem}\label{FTLTV}
Let $I\subset {K}\llbracket\x\rrbracket$ be an ideal. Then 
\[
\T_{\rm loc}(I)=\ord_t(V_{\M}(I)). 
\]
Moreover, for $(w_1,\dots,w_r)\in \T_{\rm loc}(I)\cap \R_+^r$ and $(\alpha_1,\dots,\alpha_r) \in V_K(\initial_\w(I))$, 
there exists $(\xi_1,\dots,\xi_r)\in V_\M(I)$ such that $\initial_t(\xi_i)=\alpha_it^{w_i}$ if $\alpha_i\neq 0$, 
and $\ord_t(\xi_i)>w_i$ otherwise.  
\end{Theorem}
\begin{proof}
There is a one-to-one correspondence between $\T_{\rm loc}(I+\langle x_i\rangle/\langle x_i\rangle)\subset \pR^{r-1}$ 
(resp. $V_\M(I+\langle x_i\rangle/\langle x_i\rangle)$) 
and the subset of $\T_{\rm loc}(I)$ (resp. $V_\M(I)$) consists of elements whose $i$-th component is $\infty$ (resp. 0). 
Hence it is enough to show that $\T_{\rm loc}(I)\cap \R_+^r=\ord_t(V_{\M}(I))\cap \R_+^r$. 

First, we prove that $\T_{\rm loc}(I)\supset \ord_t(V_{\M}(I))$. 
Let $\bsxi=(\xi_1,\dots,\xi_r)\in V_\M(I)$ with $\xi_i\neq 0$ for all $i$,  and set $\w=\ord(\bsxi)$. 
Let $f\in I$, and $f_0=\initial_\w(f)\in K[\x]$. 
Since the lowest order terms appearing in the expansion of $f(\bsxi)$ is $f_0(\initial_t(\bsxi))$ which should be also zero. 
As  $\initial_t(\xi_i)$'s are monomials, $f_0(\initial_t(\bsxi))=0$ is possible only if $f_0$ is not a monomial. 
Thus $\w\in \T_{\rm loc}(I)$. 

To prove the converse inclusion, 
take $\w=(w_1,\dots,w_r)\in \T_{\rm loc}(I)\cap \R_+^r$. 
Then we can define a ring homomorphism $\psi: {K}\llbracket \x \rrbracket \to T_r$, $x_i\mapsto t^{w_i}x_i$. 
Let $J$ be the ideal of $T_r$ generated by $\psi(I)$, and set $J_0=\sqrt{J}$ and $A=T_r/J_0$. 
We denote by $\overline{J}$ and $\overline{J_0}$ the image of $J\cap T_r^\circ$ and $J_0\cap T_r^\circ$ in $\overline{T}_r=K[\x]$, 
respectively. 
Then $\overline{J}=\initial_\w(I)$, and 
$\overline{J}$ and $\overline{J_0}$ have the same radical. 
Let $\boldsymbol{\alpha}=(\alpha_1,\dots,\alpha_r) \in V_K(\initial_\w(I))=V_K(\initial_\w(\overline{J_0}))$, 
and $\theta_{\boldsymbol{\alpha}}: \overline{T}_r/\overline{J_0} \to K$ the corresponding homomorphism. 
Since $K$ and $\k$ are algebraically closed, $\theta_{\boldsymbol{\alpha}}$ lifts to $\widetilde{\theta_{\boldsymbol{\alpha}}}: A \to \k$
(see \cite{FM} Theorem 3.5.3 (ii) and its proof. See also loc. cit., Corollary 3.5.7). 
By the definition of the lift, $\bsxi=(\xi_1,\dots,\xi_r)$ where $\xi_i=\widetilde{\theta_{\boldsymbol{\alpha}}}(\psi(x_i))$ satisfies the desired conditions. 
In particular, if $\boldsymbol{\alpha} \in V_{K}(\initial_\w(\overline{J_0}))\cap (K^\times)^r$, which exists as $\initial_\w(I)$ contains monomial, 
then $\w=\ord_t(\bsxi)\in \ord_t(V_\M(I))$. 
\end{proof}
\section{Irreducibility criterion for algebroid curves}
In this section, we  give an irreducibility criterion for algebroid curves over the perfect field $k$ in terms of local tropical varieties. 
\begin{Definition}\label{def of wI}
Let $I\subset k\llbracket\x\rrbracket$ be an unmixed ideal of dimensional one, and $f\in k\llbracket\x\rrbracket$. 
We define the {\it intersection number} of $f$ and $I$ as 
\[
\int(f;I)
=\ell_{k\llbracket \x\rrbracket}(k\llbracket \x\rrbracket/(I+\langle f\rangle))
=\dim_k k\llbracket\x\rrbracket/(I+\langle f \rangle)\in \pN\cup \{0\}, 
\]
where $\ell_R(M)$ denotes the length of an $R$-module $M$. 
We set 
\[
\bw(I)=(\int(x_1;I),\dots,\int(x_r;I))\in \pN^r. 
\]
\end{Definition}
As $A=k\llbracket\x\rrbracket/I$ is Cohen-Macaulay, $\int(f;I)$ coincides with the multiplicity $e(f;A)$. 
\begin{Lemma}\label{rmk on wI}
Let $I\subset k\llbracket \x\rrbracket$ be an unmixed ideal of dimension one, and set $A=k\llbracket \x\rrbracket/I$. 
Then $\int(f;I)=\sum_{\p \in \Ass I} \ell(A_{\p})\cdot \int(f;\p)$ where $\Ass I$ is the set of the associated primes of $I$. 
In particular, 
\[
\bw(I)=\sum_{\p \in \Ass I} \ell(A_{\p})\cdot \bw(\p). 
\]
\end{Lemma}
\begin{proof}
We conclude the assertion form the multiplicity formula $e(f;A)=\sum_{\p \in \Ass I} \ell(A_{\p})\cdot e(f;A_\p)$ 
(e.g. see \cite{Nagata} Corollary 23.4 or \cite{Fulton} Lemma A.2.7). 
\end{proof}
\begin{Lemma}\label{l=ord}
Let $A$ be an irreducible algebroid curve over $k$. 
Let $B=k'\llbracket s\rrbracket$ be the integral closure of $A$ with the coefficient field $k'$. 
For $\eta\in A$, $\dim_k A/\eta A=[k':k]\cdot\ord_s(\eta)$. 
\end{Lemma}
\begin{proof}
As $0\to A\to B\to B/A\to 0$ is exact and $B/A$ has finite length as an $A$-module, 
$\dim_k A/\eta A=e_A(\eta;A)=e_A(\eta;B)=\dim_k B/\eta B=[k':k]\cdot\ord_s(\eta)$. 
\end{proof}
\begin{Example}\label{ex of ord=l}
$\dim_K K\llbracket x,y\rrbracket/\langle x^3-y^4,xy^2\rangle=\dim_K K\llbracket t^4,t^3\rrbracket/\langle t^{10}\rangle =10$. 
\end{Example}
\begin{Corollary}\label{int is additive}
Let $I\subset k\llbracket \x\rrbracket$ be an unmixed ideal of dimension one, and $f,g\in K\llbracket \x\rrbracket$. Then $\int(fg;I)=\int(f;I)+\int(g;I)$. 
In particular, $\int(\x^\a; I)=\a\cdot \bw(I)$. 
\end{Corollary}

In the tropical algebraic geometry, 
a tropical variety defined by a one-dimensional prime ideal with constant coefficient is a finite union of rays 
by Bieri--Groves Theorem (\cite{BG} Theorem A, \cite{Sturmfels2} Theorem 9.6). 
We will prove that a more strong result holds for local tropical varieties; 
the local tropical variety defined by a one-dimensional prime ideal consists of a single ray.
\begin{Theorem}\label{T(p)=w(p)}
Let $\p\subset k\llbracket\x\rrbracket$ be a one-dimensional prime ideal. 
Then 
\[
\T_{\rm loc}(\p)=\ord_t(V_\M(\p))=\pR\cdot \bw(\p). 
\]
\end{Theorem}
\begin{proof}
As $\T_{\rm loc}(\p)=\T_{\rm loc}(\p K\llbracket\x\rrbracket)$, $\T_{\rm loc}(\p)=\ord_t(V_\M(\p))$ is satisfied by Theorem \ref{FTLTV}. 

We will prove that $\bw(\p)\in \ord_t(V_\M(\p))$. 
Let $A=k\llbracket\x\rrbracket/\p$, and $B$ the integral closure of $A$. 
Then $B={k'}\llbracket s\rrbracket$ where $k'\subset K$ is the coefficient field of $B$. 
Since $k$ is perfect, we may assume that $k\subset k'$ holds under the inclusion $A\subset B$. 
Let $\xi_i(s)\in k'\llbracket s\rrbracket$ be the image of $x_i$ in $A\subset B$, and set $\bsxi=(\xi_1(t),\dots,\xi_r(t))$. 
Then $\bsxi\in V_\M(\p)$, and $\bw(\p)=[k':k]\cdot \ord_t(\bsxi)\in \ord_t(V_\M(\p))$ by Lemma \ref{l=ord}. 

To complete the proof, it is enough to show that $\ord_t(V_\M(\p))\subset \pR\cdot \bw(\p)$. 
Take $0\neq \bsxi=(\xi_1,\dots,\xi_r)\in V_\M(\p)$. 
Let $A=k\llbracket\xi_1,\dots,\xi_r\rrbracket\subset \k^\circ$ and $\phi :k\llbracket \x\rrbracket\to A$, $x_i\mapsto \xi_i$. 
As $\bsxi\in V_\M(\p)$, we have $\p \subset\Ker \phi$. 
Since $\p$ is a one-dimensional prime ideal and $A$ is a domain, $\Ker \phi$ can not be larger than $\p$. 
Thus $\Ker \phi=\p$ and $A=k\llbracket \x\rrbracket/\p$. 
In particular, $A$ is of dimension one. 
Let $B=k'\llbracket s\rrbracket$ be the integral closure of $A$. 
Then $\bw(\p)=[k':k]\cdot \ord_s(\bsxi)$ by Lemma \ref{l=ord}. 
As $\k^\circ$ is integrally closed, we may assume that $B\subset \k^\circ$. 
Since $0\neq \beta\in k'$ is invertible also in $\k^\circ$, we have $\ord_t(\xi)=\ord_t(s)\cdot \ord_s(\xi)$ for $\xi\in B$. 
Therefore $\ord_t(\bsxi)=\ord_t(s)\cdot \ord_s(\bsxi)=\ord_t(s)\cdot[k':k]^{-1}\cdot\bw(\p)\in \R_+\cdot \bw(\p)$. 
\end{proof}
Using Theorem \ref{T(p)=w(p)}, we obtain an irreducibility criterion for algebroid curves in terms of local tropical variety. 
\begin{Theorem}\label{primeness}
Let $I\subset k\llbracket\x\rrbracket$ be an unmixed ideal of dimension one. 
Then the following hold. 
\begin{enumerate}
\item 
If $\T_{\rm loc}(I)\neq \pR\cdot \bw(I)$ (e.g. $\initial_{\bw(I)}(I)$ contains monomials, or $I$ has at least two tropisms) 
then $I$ is not prime. 
\item 
If $\bw(I)$ is a tropism of $I$, then $I$ is prime. 
\item 
Let $\{\w_1,\dots,\w_l\}\subset \pN^r$ be the set of the tropisms of $I$. 
If $\bw(I)\neq(\infty,\dots,\infty)$ and $\bw(I)=\w_1+\dots+\w_l$, 
then there exist prime ideals $\p_j\subset k\llbracket\x\rrbracket$ for $1\le j\le l$ 
such that $\bw(\p_j)=\w_j$ and $I=\bigcap_{j=1}^l \p_j$. 
\end{enumerate}
\end{Theorem}
\begin{proof}
Recall that, by Lemma \ref{rmk on wI}, $\bw(I)=\sum_{\p\in \Ass I} \ell(A_{\p})\cdot\bw(\p)$. 

(1) The assertion follows immediately from Theorem \ref{T(p)=w(p)}. 

(2) Since $\bw(I)\in \T_{\rm loc}(I)=\bigcup_{\p\in \Ass I} \R_+\cdot \bw(\p)$, 
there exist $\p_0\in \Ass I$ such that $\bw(I)=\gcd(\bw(\p_0))^{-1}\cdot\bw(\p_0)$. 
Since $\gcd(\bw(\p_0))^{-1}\le 1$ and $\ell(A_{\p_0})\ge 1$, 
the equality $\bw(I)=\gcd(\bw(\p_0))^{-1}\cdot\bw(\p_0)=\sum_{\p\in \Ass I} \ell(A_{\p})\cdot\bw(\p)$ is possible only if $\Ass I=\{\p_0\}$ and $\ell(A_{\p_0})=1$. 
This shows that $I=\p$. 

(3) 
Since $\w_j\in \T_{\rm loc}(I)=\bigcup_{\p\in \Ass I} \pR\cdot \bw(\p_j)$ and $\gcd(\w_j)=1$, 
there exists $\p_j\in \Ass I$ such that 
$\w_j=\gcd(\bw(\p_j))^{-1}\cdot\bw(\p_j)$. 
Hence the equality 
\[
\sum_{\p\in \Ass I} \ell(A_{\p})\cdot\bw(\p)=\bw(I)=\sum_{j=1}^l \gcd(\bw(\p_j))^{-1}\cdot\bw(\p_j)
\] 
is possible only if $\Ass I=\{\p_1,\dots,\p_l\}$ and $\ell(A_{\p_j})=\gcd(\bw(\p_j))=1$ for all $j$. 
Thus $I=\p_1\cap\dots\cap\p_l$, and $\w_j=\bw(\p_j)$ for all $j$. 
This proves the assertion. 
\end{proof}
We will present some examples. 
We use the well-known fact that if $\initial_\w(f_1),\dots, \initial_\w(f_n)$ is a regular sequence, 
then $\initial_\w(f_1,\dots,f_n)=\langle \initial_\w(f_1),\dots, \initial_\w(f_n)\rangle$. 
One may use Theorem \ref{primeness} for deciding irreducibility of plane algebroid curves. 
We also present examples of space algebroid curves. 
\begin{Example}[\cite{Ab} Kuo's Example]\label{Kuo's example}
Let $F(x,y)=(y^2-x^3)^2-x^7\in k\llbracket x,y\rrbracket$. 
Let $J=\langle F, z-(y^2-x^3-x^2y)\rangle \subset k\llbracket x,y,z\rrbracket$. 
Then $\bw(J)=(4,6,15)$. 
Note that $F$ is irreducible if and only if $J$ is prime.  
If $\ch(k)\neq 2$, then $G=2x^2yz+z^2+x^6y+x^4z\in J$ and $\initial_{(4,6,15)}(G)=2x^2yz$ is a monomial.
Hence $J$ is not prime, and thus $F$ is reducible. 
If $\ch(k)=2$, then $\bw(J)$ is a tropism of $J$ as $\initial_{(4,6,15)}(J)=\langle x^3-y^2,y^5-z^2\rangle$ contains no monomial. 
Hence $J$ is prime, and thus $F$ is irreducible. 
\end{Example}
\begin{Example}
Let $F(x,y)=y^2+x^3+xy\in k\llbracket x,y\rrbracket$. 
Then $\bw(F)=(2,3)$ and $\T_{\rm loc}(F)=\pR\cdot(1,1) \cup \pR\cdot(1,2)$. 
As $\bw(F)=(1,1)+(1,2)$, $F$ is factored as $F=F_1F_2$ with $\bw(F_1)=(1,1)$ and $\bw(F_2)=(1,2)$.  
In fact, if $k=\Q$, 
$F(x,y)=\bigl(y+\frac{x}{2}+\frac{x}{2}\sqrt{1-4x}~\bigr)\bigl(y+\frac{x}{2}-\frac{x}{2}\sqrt{1-4x}~\bigr)$, 
and if $k=\F_2=\Z/2\Z$, 
$F(x,y)=(y+x+\sum_{i=0}^\infty x^{2^i+1})(y+\sum_{i=0}^\infty x^{2^i+1})$. 
\end{Example}
\begin{Example}\label{ex primeness1}
Let 
$I=I_2\left(
	\begin{array}{ccc}
	x^3+y^2 & y & z \\
	z^2 & x & y 
	\end{array}
\right)=\langle (x^3+y^2)x-yz^2,y^2-xz,z^3-(x^3+y^2)y \rangle\subset k\llbracket x,y,z\rrbracket$. 
Then $\bw(I)=(5,6,7)$. 
Since $\initial_{(5,6,7)}((x^3+y^2)x-yz^2)=xy^2$ is a monomial, $\bw(I)\not\in\T_{\rm loc}(I)$, and thus $I$ is not prime. 
\end{Example}
\begin{Example}\label{ex 3}
Let 
$I=\langle x^3-y^2, (z^2-x^2y)^2-x^3y^2z\rangle \subset k\llbracket x,y,z\rrbracket$. 
Then $\bw(I)=(8,12,14)$ and $\initial_{\bw(I)}(I)=\langle x^3-y^2, (z^2-x^2y)^2\rangle$. 
Let $J=\langle I, u-(z^2-x^2y)\rangle \subset k\llbracket x,y,z,u\rrbracket$. 
Then $\bw(J)=(8,12,10,31)$, and $\gcd(\bw(J))=1$. 
As $\initial_{\bw(J)}(J)=\langle x^3-y^2, u^2-x^3y^2z,z^2-xy^2\rangle$ contains no monomial, 
$\bw(J)$ is a tropism of $J$. 
Hence $J$ is prime, and thus $I$ is also prime. 
\end{Example}
\begin{Example}\label{ex of primeness}
Let $a,b,c\in \N_+$ such that $a<b$, $a<c$, and $\gcd(a,b)=\gcd(a,c)=1$. 
Let $F_1=x^a+y^b+z^c$, $F_2=xy+yz+zx$, and $I=\langle F_1,F_2 \rangle\subset k\llbracket x,y,z\rrbracket$. 
Then $\bw(I)=(b+c,c+a,a+b)$, and 
\[
\T_{\rm loc}(I)=\pR\cdot(c,c,a)\cup\pR\cdot(b,a,b). 
\]
Since $\bw(I)=(c,c,a)+(b,a,b)$, there exist one-dimensional prime ideals $\p_1$ and $\p_2$ such that 
$\bw(\p_1)=(c,c,a)$, $\bw(\p)=(b,a,b)$, and $I=\p_1\cap\p_2$. 
\end{Example}
In the last section, we will show how to find the ideals $J$ in Example \ref{Kuo's example}, and \ref{ex 3}. 
\section{Value-semigroups of irreducible algebroid curves}
In the rest part of this paper, we consider only algebroid curves over the algebraically closed field $K$. 

In this section, we an algorithm for computing the value-semigroup of an irreducible algebroid curve 
$K\llbracket \x\rrbracket/\p$. 
\subsection{Numerical semigroups}
We call a subsemigroup of $\N$ a {\it numerical semigroup}. 
\begin{Definition}\label{def of semigroups}
\begin{enumerate}
(1) 
Let $A$ be an irreducible algebroid curve over $K$, and $K\llbracket t\llbracket$ the integral closure of $A$. 
We set 
\[
\Semi(A)=\{\ord_t(\eta)\mid 0\neq\eta\in A\}=\{\ell_A(A/\eta A)\mid 0\neq\eta\in A\}, 
\]
and call it the {\it value-semigroup} of $A$. 

(2) 
Let $\p\subset K\llbracket \x\rrbracket$ be a one-dimensional prime ideal. 
We set 
\[
\Semi(\p)=\{\int(f;\p)\mid f\in K\llbracket \x\rrbracket, f\not\in\p \}. 
\]

(3)
For $\w=(w_1,\dots,w_r)\in \N_+^r$, we set 
\[
\Semi(\w)=\sum_{i=1}^r \N\cdot w_i. 
\]
\end{enumerate}
\end{Definition}
Note that $\Semi(\bw(\p))\subset\Semi(\p)=\Semi(K\llbracket \x\rrbracket/\p)$ 
for a one-dimensional prime ideal $\p\subset K\llbracket \x\rrbracket$. 
It is easy to show that $\gcd(\Semi(A))=1$. 
The value-semigroup of an irreducible algebroid curve $A$ is deeply related to the singularity of $A$. 
It is known that $A$ is Gorenstein if and only if $\Semi(A)$ is symmetric (\cite{Kunz}). 
In case where $\ch(k)=0$, for an irreducible bivariate power series $F\in K\llbracket x,y\rrbracket$, 
the conductor of $\Semi(F)$ coincides with the Milnor number of $F$. 

We use Gr\"obner bases for solving membership problem for numerical semigroups. 
\begin{Definition}\label{monomial curve}
For a positive integer vector $\w=(w_1,\dots,w_r)\in \N_+^r$, 
we denote by 
\[
\Prim(\w)=\langle \x^\a-\x^\b\mid \a,\b\in \N^r, \a\cdot \w=\b\cdot \w \rangle\subset K[\x], 
\]
the kernel of the ring homomorphism $K[\x]\to K[t]$, $x_i\mapsto t^{w_i}$. 
\end{Definition}
The ideal $\Prim(\w)$ is a homogeneous prime ideal of the weighted polynomial ring $K[x_1,\dots,x_r]$ with $\deg x_i=w_i$. 
If $\Semi(\w)$ is complete intersection, a system of generators of $\Prim(\w)$ is given in  \cite{Delorme}. 

We say that a term order $\prec$ on $K[\x,t]$ is a {\it $t$-elimination oder} 
if $u\prec v$ for any two monomials $u\in K[\x]$ and $v\in K[\x,t]\backslash K[\x]$. 
The next lemma follows from the elimination theory using Gr\"obner basis (see \cite{Cox1}, \cite{Cox2}). 
\begin{Lemma}\label{is in semigroup}
Let $\w=(w_1,\dots,w_r)\in \N^r$, and $N\in \N$. 
Let $G$ be a Gr\"obner basis of $\langle x_i-t^{w_i}\mid 1\le i \le r \rangle\subset {K[t,\x]}$ 
with respect to some $t$-elimination order $\prec$. 
Then the following hold. 
\begin{enumerate}
\item 
$G\cap K[\x]$ is a Gr\"obner basis of $\Prim(\w)$ with respect to $\prec$. 
\item 
Let $u$ be the remainder of $t^N$ on the division by $(G,\prec)$. 
Then $N\in \Semi(\w)$ if and only if $u\in K[\x]$. 
Furthermore, if this is the case, $u=\x^\a$ where $\a\in \N^r$ with $N=\a\cdot\w$. 
\end{enumerate}
\end{Lemma}
The ascending chain condition holds for numerical semigroups. 
\begin{Lemma}\label{acc for semigroup}
Let $H_1\subset H_2\subset \dots$ be a sequence of numerical semigroups. 
Then there exists $i_0\in \N$ such that $H_i=H_{i_0}$ for all $i\ge i_0$. 
\end{Lemma}
\begin{proof}
Let $\preceq_{lex}$ be the lexicographic order on $\N^2$, that is, 
$(a_1,a_2)\prec_{lex}(b_1,b_2)$ if and only if $a_1<a_2$, or $a_1=b_1$ and $a_2<b_2$. 
Set $m_i=\gcd(H_i)$, and $n_i=\#((\N\cdot m_i)\backslash H_i) \in \N$. 
It is easy to see that $(m_{i+1},n_{i+1})\preceq_{lex} (m_i,n_i)$ for all $i$. 
Since $\preceq_{lex}$ is a well-ordering, 
there exists $i_0\in \N$ such that $(m_i,n_i)=(m_{i_0},n_{i_0})$ for all $i\ge i_0$. 
Therefore $H_i=H_{i_0}$ for all $i\ge i_0$. 
\end{proof}
\subsection{Local SAGBI bases}
In this section, we introduce a new notion of local SAGBI bases which is a variant of SAGBI bases 
(or canonical subalgebra bases) in local rings. 
See \cite{Sturmfels1} Chapter 11 for canonical subalgebra bases. 

Let $A$ be an algebroid curve over $K$. 
Since $K$ is algebraically closed, the integral closure of $A$ is isomorphic to $K\llbracket t\rrbracket$. 
Let $\bsxi=(\xi_1,\dots,\xi_r)$ such that $0\neq \xi_i\in tK\llbracket t\rrbracket$ for all $i$ and $A=K\llbracket \bsxi\rrbracket$. 
We set 
\begin{eqnarray*}
\phi_\bsxi:~ K\llbracket\x\rrbracket\to K\llbracket t\rrbracket, &&~x_i\mapsto \xi_i, \\
\phi_{\initial_t(\bsxi)}:~ K[\x]\to K[t], &&~x_i\mapsto \initial_t(\xi_i). 
\end{eqnarray*}
Note that $\Prim(\ord_t(\bsxi))=\sigma(\Ker\phi_{\initial_t(\bsxi)})$ for some $\sigma: K[\x]\to K[\x]$, $x_i\mapsto \alpha_i x_i$, where $\alpha_i\in K^\times$. 
In particular, $\LT_\prec(\Prim(\ord_t(\bsxi)))=\LT_\prec(\Ker\phi_{\initial_t(\bsxi)})$ for any term order $\prec$.  
\begin{Definition}
We call 
\[
K[\initial_t(A)]:=K[\initial_t(\eta)\mid \eta\in A]=K[t^i\mid i\in \Semi(A)] 
\]
the {\it initial algebra} of $A$. 
We say that $\bsxi$ is a {\it local SAGBI basis} of $A$ 
if $K[\initial_t(A)]=K[\initial_t(\bsxi)]$, 
in other words, $\Semi(A)=\Semi(\ord_t(\bsxi))$. 
\end{Definition}
\begin{Remark}
It holds that $K[\initial_t(\bsxi)]\subset K[\initial_t(A)]$, but the equality does not hold in general. 
Since $\Semi(A)$ is finitely generated as a semigroup, there exists a finite local SAGBI basis of $A$. 
The assumption that $A$ is of dimension one is essential for the existence of finite local SAGBI bases. 
\end{Remark}
\begin{Proposition-Definition}[local reduction]\label{local reduction}
Let $\eta\in K\llbracket t \rrbracket$. 
We set $\w=\ord_t(\bsxi)$. 
Then there exist $q\in K\llbracket \x \rrbracket$ and $\zeta\in K\llbracket \bsxi\rrbracket$ satisfying the following: 
\begin{enumerate}
\item $\eta=q(\bsxi)+\zeta$. 
\item $\initial_t(\zeta)\not\in K[\initial_t(\bsxi)]$ if $\zeta\neq 0$. 
\item $\ord_t(\eta)=\ord_\w(q)$ if $q\neq 0$. 
\end{enumerate}
We call $f$ a quotient, and $\zeta$ a remainder of $\eta$ on local reduction by $\bsxi$. 
\end{Proposition-Definition}
\begin{proof}
We define $\eta_i\in K\llbracket \bsxi\rrbracket$, $\alpha_i\in K$, and $\a_i\in \N^r$ inductively on $i$ in the following manner: 
Set $\eta_0=\eta$. 
If $\initial_t(\eta_i)\in K[\initial_t(\bsxi)]$, take ${\c_i}\in\N^r$ and $\beta_i\in K$ such that 
$\ord_t(\eta_i)=\ord_t(\bsxi^{\c_i})$ and $\initial_t(\eta_i)=\beta_i\initial_t(\bsxi^{\c_i})$. 
We set $\eta_{i+1}=\eta_i-\beta_i \bsxi^{\c_i}$. 

If $\eta_m\not\in K[\initial_t(\bsxi)]$ or $\eta_m=0$ for some $m\in \N$, 
then $q:=\sum_{i=0}^{m-1} \beta_i\x^{\c_i}$ and $\zeta:=\eta_m$ satisfy the desired conditions. 
If $0\neq \eta_i\in K[\initial_t(\bsxi)]$ for all $i$, 
then $q:=\sum_{i=0}^{\infty} \beta_i\x^{\c_i}$ and $\zeta:=0$ satisfy the desired conditions. 
\end{proof}
If $q\neq 0$, then $\initial_t(\eta)=\initial_\w(q)(\initial_t(\bsxi))$ by Proposition-Definition \ref{local reduction} (3). 
In general, it takes infinite time to compute local reduction. 
In the special case where $\gcd(\ord_t(\bsxi))=1$, the remainder of $\eta\in K\llbracket t\rrbracket$ on local reduction by $\bsxi$ is zero if $\ord_t(\eta)\ge c$ 
where $c$ is the conductor of $\Semi(\ord_t(\bsxi))$. 
Hence one can compute the remainder of any element on local reduction by $\bsxi$ in finite time in this case. 
\begin{Theorem}\label{initial in ker}
It holds that $\initial_\w(\Ker \phi_\bsxi)\subset \Ker \phi_{\initial_t(\bsxi)}$ and $\sqrt{\initial_\w(\Ker \phi_\bsxi)}= \Ker \phi_{\initial_t(\bsxi)}$ 
where $\w=\ord_t(\bsxi)$. 
\end{Theorem}
\begin{proof}
Let $f\in \Ker \phi_\bsxi$, and $f_0=\initial_\w(f)\in K[\x]$. 
Since the lowest order terms appearing in the expansion of $f(\bsxi)$ is $f_0(\initial_t(\bsxi))$ which should be also zero. 
Hence $f_0\in \Ker \phi_{\initial_t(\bsxi)}$, and thus $\initial_\w(\Ker \phi_\bsxi)\subset \Ker \phi_{\initial_t(\bsxi)}$. 

We will prove that $\sqrt{\initial_\w(\Ker \phi_\bsxi)}= \Ker \phi_{\initial_t(\bsxi)}$. 
Set $\p=\Ker \phi_\bsxi$. 
Then $V_K(\initial_\w(\p))=(V_{K}(\initial_\w(\p))\cap (K^\times)^r)\cup \{0\}$, 
since $0\neq (\alpha_1,\dots,\alpha_r) \in V_K(\initial_\w(\p))$ with $\alpha_i=0$ for some $i$ induces 
$\bseta\in V_\M(\p)$ such that $\ord_t(\bseta)\not\in \pR\cdot \w$ by Theorem \ref{FTLTV}. 

Take any $0\neq \bseta\in V_\M(\p)$ and consider the defining ideal of $K[\initial_t(\bseta)]$. 
Let $K\llbracket s\rrbracket \subset \k^\circ$ be the integral closure of $K\llbracket \bseta\rrbracket$ in $\k^\circ$. 
As we saw in the proof of Theorem \ref{T(p)=w(p)}, 
$\ord_t(f(\bseta))=\ord_t(s)\cdot \ord_s(f(\bseta))=\ord_t(s)\cdot \int(f;\p)$, and 
$\ord_t(\bseta)=\ord_t(s)\cdot \w$ for $f\in K\llbracket \x\rrbracket$. 
Therefore $\initial_\w(f)(\initial_t(\bseta)=0$ if and only if $\ord_\w(f)>\int(f;\p)$. 
Hence the defining ideal of $K[\initial_t(\bseta)]$ is independent form choice of $\bseta\in V_\M(\p)$, 
and thus coincides with $\Ker \phi_{\initial_t(\bsxi)}$ as $\bsxi\in V_\M(\p)$. 
For $(\alpha_1,\dots,\alpha_r) \in V_{K}(\initial_\w(\p))\cap (K^\times)^r$, there exists $\bseta\in V_\M(\p)$ such that 
$\initial_t(\bseta)=(\alpha_1t^{w_1},\dots,\alpha_rt^{w_r})$ by Theorem \ref{FTLTV}. 
This implies $(\alpha_1,\dots,\alpha_r) \in V_K(\Ker \phi_{\initial_t(\bsxi)})$. 

Therefore we conclude that $V_K(\initial_\w(\p))\subset  V_K(\Ker \phi_{\initial_t(\bsxi)})$. 
By Hilbert's Nullstellensatz, this implies that $\sqrt{\initial_\w(\Ker \phi_\bsxi)}\supset \Ker \phi_{\initial_t(\bsxi)}$. 
Hence $\sqrt{\initial_\w(\Ker \phi_\bsxi)}= \Ker \phi_{\initial_t(\bsxi)}$. 
\end{proof}
\begin{Proposition}\label{red 0}
Assume that $\bsxi=(\xi_1,\dots,\xi_r)$ is a local SAGBI basis of $A$. 
For $\eta\in K\llbracket t\rrbracket$, 
$\eta\in A$ if and only if some (any) remainder of $\eta$ on local reduction by $\bsxi$ is zero. 
In particular, if $\bsxi$ is a local SAGBI basis of $A$, then $K\llbracket\bsxi\rrbracket=A$. 
\end{Proposition}
\begin{proof}
The assertion follows obviously from the definition of local SAGBI bases and remainders of local reduction. 
\end{proof}
\begin{Lemma}\label{lift ker}
Let $\w=\ord_t(\bsxi)$, and 
$f \in\Ker(\phi_{\initial_t(\bsxi)})$ a homogeneous polynomial with respect to $\w$. 
If one can take zero as a remainder of $f(\bsxi)$ with the quotient $g$ on local reduction by $\bsxi$, 
then $f-g\in \Ker \phi_\bsxi$ and $f\in \initial_{\w}(\Ker \phi_\bsxi)$. 
\end{Lemma}
\begin{proof}
By Proposition-Definition \ref{local reduction} (1), $f(\bsxi)=g(\bsxi)$ and thus $f-g\in \Ker \phi_\bsxi$. 
By Proposition-Definition \ref{local reduction} (3) and $f \in\Ker(\phi_{\initial_t(\bsxi)})$, we have $\ord_\w(g)=\ord_t(f(\bsxi))>\ord_\w(f)$. 
Hence $f=\initial_\w(f-g)\in\initial_{\w}(\Ker \phi_\bsxi)$. 
\end{proof}
Similarly to SAGBI bases (\cite{Sturmfels1} Theorem 11.4), an analogue of Buchberger's criterion holds for local SAGBI basis. 
\begin{Theorem}\label{SB criterion}
Let $A=K\llbracket\bsxi\rrbracket\subset K\llbracket t\rrbracket$ be an irreducible algebroid curve over $K$, 
and set $\w=\ord_t(\bsxi)$. 
Let $G$ be a system of binomial generators of $\Ker(\phi_{\initial_t(\bsxi)})$. 
Then the following are equivalent: 
\begin{enumerate}
\item 
$\bsxi$ is a local SAGBI basis of $A$. 
\item 
Any remainder of $f(\bsxi)$ on local reduction by $\bsxi$ is zero for all $f\in K\llbracket \x \rrbracket$. 
\item 
One can take zero as a remainder of $f(\bsxi)$ on local reduction by $\bsxi$ for all $f\in G$. 
\item 
$\initial_\w(\Ker\phi_\bsxi)=\Ker(\phi_{\initial_t(\bsxi)})$. 
\end{enumerate}
\end{Theorem}
\begin{proof}
(1) $\Rightarrow$ (2) immediately follows from the definition of local SAGBI bases and local reduction. 
(2) $\Rightarrow$ (3) is obvious. 
(3) $\Rightarrow$ (4) follows from Lemma \ref{initial in ker} and Lemma \ref{lift ker}. 

(4) $\Rightarrow$ (1): 
Let $\eta\in A$, and write $f(\bsxi)=\eta$ for $f\in K\llbracket \x\rrbracket$. 
Then $\ord_\w(f)\le \ord_t(\eta)$. 
Hence we may take $f\in K\llbracket \x\rrbracket$ so that $\ord_\w(f)$ is maximal among the power series satisfying $f(\bsxi)=\eta$. 
We claim that $\initial_\w(f)\not\in \Ker(\phi_{\initial_t(\bsxi)})$. 
Assume, to the contrary, that $\initial_\w(f)\in \Ker(\phi_{\initial_t(\bsxi)})$. 
By the assumption $\Ker(\phi_{\initial_t(\bsxi)})=\initial_\w(\Ker\phi_\bsxi)$, 
there exists $g\in \Ker\phi_\bsxi$ such that $\initial_\w(f)=\initial_\w(g)$. 
Then $\ord_\w(f-g)>\ord_\w(f)$ and $(f-g)(\bsxi)=f(\bsxi)-g(\bsxi)=f(\bsxi)=\eta$. 
This is a contradiction. 
Thus $\initial_\w(f)\not\in \Ker(\phi_{\initial_t(\bsxi)})$. 
Therefore $\initial_t(\eta)=\initial_t(f(\bsxi))=\initial_\w(f)(\initial_t(\bsxi))\in K[\initial_t(\bsxi)]$. 
This proves that $\bsxi$ is a local SAGBI basis of $A$. 
\end{proof}
\begin{Remark}\label{implementation}
Even if all the components of $\bsxi$ are polynomials, $\Ker\phi_\bsxi$ is not generated by polynomials in general. 
For example, let $\bsxi=(t+t^2, t^2+t^3)$. 
The defining ideal of $K[t+t^2, t^2+t^3]$ is generated by $F(x,y)=y^2-x^3+xy$. 
If one regard $F$ as a power series, then $\bw(F)=(2,3)$, and $F$ has two tropisms $(1,1)$ and $(1,2)$. 
Thus $F$ is factors as $F=F_1F_2$ with $\bw(F_1)=(1,2)$ and $\bw(F_2)=(1,1) $, and $F_1$ is the generator of $\Ker\phi_\bsxi$. 
As $F_1$ is not divisible by $F$, $F_1$ is not a polynomial. 

One can compute a local SAGBI basis of $K\llbracket \bsxi\rrbracket$ in the following manner: 
Let $G$ be a system of binomial generators of $\Ker(\phi_{\initial_t(\bsxi)})$, 
and compute the remainder of $f(\bsxi)$ on local reduction by $\bsxi$ for all $f\in G$ parallelly. 
If the remainder of  $f(\bsxi)$ is zero for all $f\in G$, then $\bsxi$ is a local SAGBI basis of $K\llbracket \bsxi\rrbracket$. 
If there exists a non-zero remainder $\eta$ of $f(\bsxi)$ for some $f\in G$, 
replace $\bsxi$ by $(\bsxi,\eta)$. Then $\Semi(\ord_t(\bsxi))$ become strictly larger. 
This can be done in finite time. If $\gcd(\ord_t(\bsxi))=1$, then the remainders on local reduction by $\bsxi$ is computable in finite time, 
and if $\gcd(\ord_t(\bsxi))\neq 1$, then $\Ker \phi_{\initial_t(\bsxi)}\neq \initial_\w(\Ker \phi_\bsxi)$ and thus eventually one obtain a non-zero 
remainder of  $f(\bsxi)$ for some $f\in G$. 
Repeating this procedure, one eventually obtain a local SAGBI basis of $K\llbracket \bsxi\rrbracket$ by Lemma \ref{acc for semigroup}. 
\end{Remark}
We will give a variant of the above procedure for computing local SAGBI bases which we need in the next subsection. 
\begin{Algorithm}[Algorithm for computing local SAGBI bases]\label{algorithm for local SAGBI basis}
The following algorithm computes a local SAGBI basis of an irreducible algebroid curve $A=K\llbracket \bsxi\rrbracket \subset K\llbracket t\rrbracket$ 
where $K\llbracket t\rrbracket$ is the integral closure of $A$. \\
{\bf Input:} $\bsxi=(\xi_1,\dots,\xi_r)$, $0\neq \xi_i\in tK\llbracket t\rrbracket$. \\
{\bf Output:} $\bsxi'=(\xi_1,\dots,\xi_{r'})$, $r'\ge r$, a local SAGBI basis of $K\llbracket \bsxi\rrbracket$. 
\begin{enumerate}[1{:~}]
\item $\x$\subst $(x_1,\dots,x_r)$, $\w$\subst $\ord_t(\bsxi)$. 
\item \pscwhile{$\initial_\w(\Ker \phi_\bsxi)\neq\Ker \phi_{\initial_t(\bsxi)}$}. 
\item \quad Take $\x^\a-\alpha\x^\b\in \Ker \phi_{\initial_t(\bsxi)}\backslash \initial_\w(\Ker \phi_\bsxi)$, $N$\subst ${\ord_t(\bsxi^\a-\alpha\bsxi^\b)}$, 
			$\eta$\subst $\bsxi^\a-\alpha\bsxi^\b$. 
\item \quad \pscwhile{$N\in \Semi(\w)$}
\item \quad\quad Take $\c\in \N^r$ and $\beta\in K^\times$ such that $\initial_t(\eta)=\beta\initial_t(\bsxi^\c)$. 
\item \quad\quad $N$\subst ${\ord_t(\eta-\beta \bsxi^\c)}$, $\eta$\subst $\eta-\beta \bsxi^\c$. 
\item \quad \pscendwhile
\item \quad $\bsxi$\subst $(\bsxi,\eta)$, $\x$\subst $(\x,x_{r+1})$, $\w$\subst $(\w,N)$, $r$\subst $r+1$. 
\item \pscendwhile
\item \pscreturn $\bsxi$ 
\end{enumerate}
\end{Algorithm}
\begin{proof}
We give a proof of the correctness and of this algorithm. 
The while loop from Line 4 to 7 computes the local reduction of $\bsxi^\a-\alpha\bsxi^\b$ by $\bsxi$. 
The remainder of $\bsxi^\a-\alpha\bsxi^\b$ on local reduction by $\bsxi$ is not zero by Lemma \ref{lift ker}.  
Thus the while loop from Line 4 to 7 terminates in finite time. 
When the while loop from Line 2 to 9 completes, the semigroup $\Semi(\w)$ become larger. 
Hence this while loop terminates in finite time by Lemma \ref{acc for semigroup}. 
Therefore we eventually have $\initial_\w(\Ker \phi_\bsxi)=\Ker \phi_{\initial_t(\bsxi)}$, 
and thus the output of this algorithm is a local SAGBI basis of $A$ by Theorem \ref{SB criterion}. 
\end{proof}
\subsection{Algorithm for computing value-semigroups of irreducible algebroid curves}
Let $\p\subset K\llbracket \x\rrbracket$ be a one-dimensional prime ideal. 
The purpose of this subsection is to give an algorithm for computing $\Semi(\p)$. 

Since we can write $A:=K\llbracket \x\rrbracket/\p \cong K\llbracket \bsxi\rrbracket\subset K\llbracket t\rrbracket$, 
where $K\llbracket t\rrbracket$ is the integral closure of $A$, 
one can compute $\Semi(A)=\Semi(\p)$ by computing a local SAGBI basis of $K\llbracket \bsxi\rrbracket$. 
However, there is a problem that $\bsxi$ is not easy to compute, and has infinitely many terms in general. 
We will present an algorithm for computing the preimage of a local SAGBI basis of $K\llbracket \bsxi\rrbracket$ in $K\llbracket \x\rrbracket$ without computing $\bsxi$. 
To do this, it is enough to give a method to compute the local reduction of $\eta=f(\bsxi)\in K\llbracket t\rrbracket$ by $\bsxi$ without knowing $\bsxi$. 
\begin{Observation}\label{without bsxi}
In each step on local reduction, for $h(\bsxi)\in K\llbracket t\rrbracket$, we have to find $\c\in\N^r$ and $\beta\in K^\times$ such that 
$\ord_t(h(\bsxi))=\c\cdot\ord_t(\bsxi)$ and $\initial_t(h(\bsxi))=\beta \initial_t(\bsxi^\c)$. 
By Lemma \ref{l=ord}, $\int(h;\p)=\ord_t(h(\bsxi))$ and $\ord_t(\bsxi)=\bw(\p)$. 
Hence $\ord_t(h(\bsxi))=\c\cdot\ord_t(\bsxi)$ is equivalent to 
$\int(h;\p)=\c\cdot\bw(\p)$, and thus 
one can find the $\c$ without knowing $\bsxi$. 
The $\beta$ is the unique element satisfying $\ord_t(h(\bsxi)-\beta\bsxi^\c)>\ord_t(h(\bsxi))$, equivalently 
$\int(h-\beta \x^\c;\p)>\int(h;\p)$. 
One can find this $\beta\in K^\times$ by computing a standard basis of $I+\langle f-a\x^c\rangle$ with a parameter $a$ 
in a similar way to comprehensive Gr\"obner basis \cite{W} 
(in the special case, one can compute $\beta$ more easily. See Observation \ref{matrix method}). 
Hence one can find $\c$ and $\beta$ without knowing $\bsxi$. 
\end{Observation}
\begin{Algorithm}[Algorithm for computing value-semigroups]\label{Semi(p)}
The following algorithm computes $\Semi(\p)$ for a one-dimensional prime ideal $\p$. \\
{\bf Input:} a one-dimensional prime ideal $\p\subset K\llbracket \x \rrbracket=K\llbracket x_1,\dots,x_r \rrbracket$ with $x_i\not\in \p$ for all $i$. \\
{\bf Output:} ($\p'$, $\w$) where $\p'\subset K\llbracket x_1,\dots,x_{r'} \rrbracket$, $r'\ge r$, and $\w=\bw(\p')$ 
such that $K\llbracket \x\rrbracket/\p\cong K\llbracket x_1,\dots,x_{r'} \rrbracket/\p'$ and $\Semi(\p)=\Semi(\w)$. 
\begin{enumerate}[1{:~}]
\item $\x$\subst $(x_1,\dots,x_r)$, $\w$\subst $\bw(\p)$. 
\item \pscwhile{$\initial_\w(\p)\neq \sqrt{\initial_\w(\p)}$}. 
\item \quad Take $f=\x^\a-\alpha \x^\b\in \sqrt{\initial_\w(\p)}\backslash \initial_\w(\p)$,  $N$\subst ${\int(f;\p)}$. 
\item \quad \pscwhile{$N \in \Semi(\w)$}
\item \quad\quad Take $\c\in \N^r$ such that $N=\w\cdot \c$. 
\item \quad\quad Take the unique $\beta\in K^\times$ such that $\int(f-\beta \x^\c;\p)>\int(f;\p)$. 
\item \quad\quad $N$\subst ${\int(f-\beta \x^\c;\p)}$, $f$\subst $f-\beta \x^\c$. 
\item \quad \pscendwhile
\item \quad $\x$\subst $\p$\subst $\langle \p, x_{r+1}-f\rangle$, $(\x,x_{r+1})$, $\w$\subst $(\w,N)$, $r$\subst $r+1$. 
\item \pscendwhile
\item \pscreturn ($\p$, $\w$). 
\end{enumerate}
\end{Algorithm}
\begin{proof}
Set $A:=K\llbracket \x\rrbracket/\p \cong K\llbracket \bsxi\rrbracket\subset K\llbracket t\rrbracket$ 
where $K\llbracket t\rrbracket$ is the integral closure of $A$. 
Then $\p=\Ker \phi_\bsxi$ and  $\sqrt{\initial_\w(\p)}=\initial_\w(\Ker \phi_\bsxi)$ by Theorem \ref{initial in ker}. 
The ideal $\p$ changes only at Line 9, and the residue class ring of $\p$ does not change. 
Note that $\Ker \phi_{\bsxi'}=\langle \p, x_{r+1}-f\rangle$ where $\bsxi'=(\bsxi, f(\bsxi))$. 
Hence, by Observation \ref{without bsxi}, 
each step of this algorithm completely corresponds to that of Algorithm \ref{algorithm for local SAGBI basis} with $\bsxi$ as the input. 
Thus $\Semi(\p)=\Semi(\p')=\Semi(\bw(\p'))$ is satisfied for the output $\p'$ by the definition of local SAGBI basis. 
\end{proof}
Note that one can decide whether $N\in \Semi(\w)$ or not, and find $\c\in \N^r$ such that $N=\c\cdot\w$ if $N\in \Semi(\w)$ at the same time 
by using Lemma \ref{is in semigroup}.  

It is not hard to show that if $\initial_\w(I)$ is prime for some $\w\in \N_+^r$, then $I$ is also prime. 
Thus one can use Algorithm \ref{Semi(p)} to obtain an evidence of the primeness of a prime ideal. 
\begin{Example}
Let $a_1,a_2\in \N_+$ with $\gcd(a_1,a_2)=1$. 
Take $2\le d\in \N_+$ and $b=c_1a_1+c_2a_2\in \Semi(a_1,a_2)$, $c_1, c_2\in \N_+$, such that $b \not\in \Semi(da_1,da_2)$ and $b>da_1a_2$. 
Let $F(x,y)=(y^{a_1}-x^{a_2})^d-x^{c_1}y^{c_2}\in K\llbracket x,y\rrbracket$. 
Then $\bw(F)=(da_1,da_2)$, and $\initial_{\bw(F)}(F)=(y^{a_1}-x^{a_2})^d$. 
Let $g=y^{a_1}-x^{a_2}$. Then 
\begin{eqnarray*}
\int(g;F) &=&\dim_K K\llbracket x,y \rrbracket/\langle y^{a_1}-x^{a_2}, x^{c_1}y^{c_2}\rangle
=\dim_K K\llbracket t^{a_1},t^{a_2}\rrbracket /\langle t^{c_1a_1+c_2a_2}\rangle \\
&=&c_1a_1+c_2a_2=b\not\in \Semi(da_1,da_2). 
\end{eqnarray*}
Let $J=\langle F, z-g\rangle =\langle (y^{a_1}-x^{a_2})^d-x^{c_1}y^{c_2}, z-(y^{a_1}-x^{a_2}) \rangle =
\langle z^d-x^{c_1}y^{c_2}, (y^{a_1}-x^{a_2})-z \rangle\subset K\llbracket x,y,z\rrbracket$. 
Then $\bw(J)=(da_1,da_2,b)$. 
Since $\initial_{\bw(J)}(J)=\langle z^d-x^{c_1}y^{c_2}, y^{a_1}-x^{a_2} \rangle=\Prim(da_1,da_2,b)$, 
$F$ is irreducible and $\Semi(F)=\Semi(da_1,da_2,b)$. 
\end{Example}
\begin{Example}
Let $a,b,c\in \N_+$ such that $c$ is odd, and $20<4a+6b+5c$. Let 
\[
I=\langle x^3-y^2, (z^2-xy)^2-x^ay^bz^c\rangle \subset K\llbracket x,y,z\rrbracket. 
\]
Then $\bw(I)=(8,12,10)$, $\initial_{\bw(I)}(I)=\langle x^3-y^2, (z^2-xy)^2\rangle$. 
Let $g=z^2-xy$. Then $\int(g;I)=4a+6b+5c\not\in \Semi(8,12,10)$ as $4a+6b+5c$ is odd. 
Let $J=\langle I, u-g\rangle \subset K\llbracket x,y,z,u\rrbracket$. 
Then $\bw(J)=(8,12,10,4a+6b+5c)$, and $\gcd(\bw(J))=\gcd(2,5c)=1$. 
Since it holds that $\initial_{\bw(J)}(J)=\langle x^3-y^2, u^2-x^ay^bz^c,z^2-xy\rangle=\Prim(\bw(J))$, 
we conclude that $I$ is prime and $\Semi(I)=\Semi(8,12,10,4a+6b+5c)$. 
\end{Example}
If one apply Algorithm \ref{Semi(p)} to a non-prime ideal $I$, then some errors would occur. 
For example, the uniqueness of $\beta$ in Line 6 fails. 
One may regard it as an evidence of the non-primeness of $I$, 
and this lead us to an algorithm for deciding irreducibility of algebroid curves. 
\section{Algorithm for deciding irreducibility of reduced algebroid curves}
We will present an algorithm for proving irreducibility of algebroid curves over an algebraically closed field. 
\subsection{Parametric intersection numbers}
Let $I\subset K\llbracket \x\rrbracket$ be an unmixed ideal of dimension one, 
and take $f,g\in K\llbracket \x\rrbracket$ such that $\inf(f;I)=\int(g;I)$. 
We will investigate how the value $\int(f-\alpha g;I)$ varies as $\alpha\in K^{\times}$ changes. 
If  $f$ and $g$ are polynomials and $I$ is generated by polynomials, 
one can compute this by using comprehensive Gr\"obner basis \cite{W} and Lazard's homogenization technique \cite{Lazard}. 

In case where $I$ is prime, we obtain the next lemma by Lemma \ref{l=ord}. 
\begin{Lemma}\label{prime case parametric}
Let $\p\subset K\llbracket \x\rrbracket$ be a one-dimensional prime ideal, and $f,g\in K\llbracket \x\rrbracket$ with $\int(f;\p)=\int(g;\p)<\infty$. 
Then there exists $\alpha\in K^\times$ such that $\int(f-\alpha g;\p)>\int(f;\p)$ and $\int(f-\beta g;\p)=\int(f;\p)$ for $\beta\neq \alpha$. 
\end{Lemma}
One can use Lemma \ref{prime case parametric} for testing a one-dimensional ideal to be prime. 
If $I$ is decided to be not prime by Lemma \ref{prime case parametric}, 
one can construct $J$ such that its residue class ring is isomorphic to that of $I$, and $J$ has at least two tropisms. 
\begin{Theorem}\label{test 1}
Let $I\subset K\llbracket \x\rrbracket$ be an unmixed ideal of dimension one, 
and $f,g\in K\llbracket\x\rrbracket$ with $\int(f;I)=\int(g;I)<\infty$. 
Then the following holds. 
\begin{enumerate}
\item 
$\int(f-\alpha g;I)\le \int(f;I)$ for generic $\alpha\in K^\times$. 
\item 
If $\int(f-\alpha g;I)<\int(f;I)$ for generic $\alpha\in K^\times$, 
then 
\[
J=\langle I,x_{r+1}-f,x_{r+2}-g\rangle \subset K\llbracket x_1,\dots,x_r,x_{r+1},x_{r+2}\rrbracket
\]
has at least two tropisms. 
\item 
If $\int(f-\alpha g;I)=\int(f;I)$ for generic $\alpha\in K^\times$ and there exist $\beta_1,\beta_2\in K^\times$, $\beta_1\neq\beta_2$, 
such that $\int(f-\beta_ig;I)>\int(f;I)$ for $i=1,2$, then 
\[
J=\langle I,x_{r+1}-(f-\beta_1g), x_{r+1}-(f-\beta_2g)\rangle \subset K\llbracket x_1,\dots,x_r,x_{r+1},x_{r+2}\rrbracket
\]
has at least two tropisms. 
\item 
If $h\not\in \sqrt{I}$ and $\int(h;I)=\infty$, 
then 
\[
J=\langle I,x_{r+1}-h\rangle \subset K\llbracket x_1,\dots,x_r,x_{r+1}\rrbracket
\]
has at least two tropisms. 
\end{enumerate}
\end{Theorem}
\begin{proof}
Let $I=\q_1\cap\dots\cap\q_l$ be the irredundant primary decomposition, and $\p_j=\sqrt{\q_j}$. 
Note that $\{(\gcd\bw(\p_j))^{-1} \bw(\p_j)\mid 1\le j\le l\}$ is the set of the tropism of $I$ by Theorem \ref{T(p)=w(p)} and Lemma \ref{intersection of ideals}. 
Set $\w=\bw(I)$. 
Let $A=K\llbracket \x\rrbracket/I$ and $A_j=K\llbracket \x\rrbracket/\p_j$. 
We identify the integral closure of $A_j$ with $K\llbracket t\rrbracket$ for all $1\le j\le l$. 
Let $\phi_j:K\llbracket \x\rrbracket\to K\llbracket t\rrbracket$ be 
the composition of the natural surjection $K\llbracket \x\rrbracket\to A_j$ and the inclusion $A_j\hookrightarrow K\llbracket t\rrbracket$. 
Recall that 
\[
\int(h;I)=\sum_{j=1}^l \ell(A_{\p_j}) \cdot\ord_t(\phi_j(h))
\]
for $h\in K\llbracket \x\rrbracket$ by Lemma \ref{rmk on wI} and Lemma \ref{l=ord}. 
We set $u_j=\ord_t(\phi_j(f))$ and $v_j=\ord_t(\phi_j(g))$. 
Then 
\[
\sum_{j=1}^l \ell(A_{\p_j}) \cdot u_j=\int(f;I)=\int(g;I)=\sum_{j=1}^l \ell(A_{\p_j}) \cdot v_j. 
\]
Hence, for generic $\alpha\in K^\times$, it holds that 
\[
\int(f-\alpha g;I)=\sum_{j=1}^l \ell(A_{\p_j}) \cdot \ord_t(\phi_j(f)-\alpha\phi_j(g))=\sum_{j=1}^l \ell(A_{\p_j}) \cdot \min\{u_j,v_j\}. 
\]
Therefore $\int(f-\alpha g;I)=\int(f;I)$ holds for generic $\alpha\in K^\times$ if and only if $u_j=v_j$ for all $1\le j\le l$. 

(1) As $\int(f-\alpha g;I)=\sum_{j=1}^l \ell(A_{\p_j}) \cdot \min\{u_j,v_j\}$, it holds that 
$\int(f-\alpha g;I)\le \int(f;I)$ for generic $\alpha\in K^\times$. 

(2) Since $\int(f-\alpha g;I)<\int(f;I)$ for generic $\alpha\in K^\times$, 
there exists $1\le j\le l$ such that $u_j\neq v_j$. 
Moreover, since $\sum_{j=1}^l \ell(A_{\p_j}) \cdot u_j=\sum_{j=1}^l \ell(A_{\p_j}) \cdot v_j$, 
there exist $1\le j_1, j_2\le l$ such that $u_{j_1}> v_{j_1}$ and $u_{j_2}<v_{j_2}$. 
For $1\le j\le l$, we set 
\[
\p_j' = \langle \p_j,x_{r+1}-f,x_{r+2}-g\rangle \subset K\llbracket x_1,\dots,x_r, x_{r+1}, x_{r+2}\rrbracket. 
\]
Then $\p_1',\dots,\p_l'$ are the associated prime ideals of $J$, and $\bw(\p_j')=(\bw(\p_j),u_j,v_j)$. 
Since $u_{j_1}>v_{j_1}$ and $u_{j_2}<v_{j_2}$, we have 
$\R_+ \cdot \bw(\p_{j_1}')\neq \R_+\cdot \bw(\p_{j_2}')$. 

(3) Since $\int(f-\alpha g;I)=\int(f;I)$ for generic $\alpha\in K^\times$, it holds that $u_j=v_j$ for all $1\le j\le l$. 
Thus, for each $1\le j\le l$, there exists $\alpha_j\in K^\times$ such that $\phi_j(f)=\alpha_j\phi_j(g)$. 
Then $\ord_t(\phi_j(f)-\alpha_j\phi_j(g))>u_j$, and $\ord_t(\phi_j(f)-\beta \phi_j(g))=u_j$ for $\beta\neq \alpha_j$. 
Therefore $\int(f-\beta g;I)>\int(f;I)$ if and only if $\beta=\alpha_j$ for some $1\le j\le l$. 
Hence there exist $1\le j_1,j_2\le l$ such that $\beta_1=\alpha_{j_1}$ and $\beta_2=\alpha_{j_2}$. 
For $1\le j\le l$, we set 
\[
\p_j' = \langle \p_j,x_{r+1}-(f-\beta_1g), x_{r+1}-(f-\beta_2g)\rangle \subset K\llbracket x_1,\dots, x_r, x_{r+1} \rrbracket. 
\]
Then $\p_1',\dots,\p_l'$ are the associated prime ideals of $J$, and 
\begin{eqnarray*}
\bw(\p_{j_1}') &=& (\bw(\p_{j_1}), \ord_t(\phi_{j_1}(f)-\alpha_{j_1}\phi_{j_1}(g)),~ \ord_t(\phi_{j_1}(f))), \\
\bw(\p_{j_2}') &=& (\bw(\p_{j_2}), \ord_t(\phi_{j_2}(f)),~ \ord_t(\phi_{j_2}(f)-\alpha_{j_2}\phi_{j_2}(g))).  
\end{eqnarray*}
As $\ord_t(\phi_{j_1}(f)-\alpha_{j_1}\phi_{j_1}(g))> \ord_t(\phi_{j_1}(f))$ and 
$\ord_t(\phi_{j_2}(f))<\ord_t(\phi_{j_2}(f)-\alpha_{j_2}\phi_{j_2}(g))$, 
we have 
$\R_+ \cdot \bw(\p_{j_1}')\neq \R_+\cdot \bw(\p_{j_2}')$. 

(4) Since $h\not\in \sqrt{I}=\p_1\cap\dots \cap\p_\ell$, there exists $1\le j_1 \le \ell$ such that $h\not\in \p_{j_1}$. 
On the other hand, as $\int(h;I)=\infty$, there there exists $1\le j_2 \le \ell$ such that $h\in \p_{j_2}$. 
For $1\le j\le l$, we set 
\[
\p_j' = \langle \p_j,x_{r+1}-h \rangle \subset K\llbracket x_1,\dots,x_r, x_{r+1}\rrbracket. 
\]
Then $\p_1',\dots,\p_l'$ are the associated prime ideals of $J$, and $\bw(\p_j')=(\bw(\p_j),\ord_t(\phi_j(h)))$. 
Since $h\not\in \p_{j_1}$ and $h\in \p_{j_2}$, it holds that $\ord_t(\phi_{j_1}(h))<\infty$ and $\ord_t(\phi_{j_2}(h))=\infty$. 
Thus $\R_+ \cdot \bw(\p_{j_1}')\neq \R_+\cdot \bw(\p_{j_2}')$. 
\end{proof}
\begin{Observation}\label{matrix method}
Let $A=K\llbracket \x\rrbracket/I$ be a reduced algebroid curve. 
Assume $n:=\int(x_1;I)<\infty$. 
Since $A$ is Cohen-Macaulay, is a free $K\llbracket x_1\rrbracket$-module of rank $n$.  
Let $\prec$ be a term order, and $\w\in \N_+^r$. 
We may regard $\Gamma=\{ \x^\a\not\in \initial_\prec(\initial_\w(I+\langle x_1\rangle))\}$ as a $K\llbracket x_1\rrbracket$-basis of $A$. 
For $f\in K\llbracket \x\rrbracket$, 
we denote by $M_f\in M_n(K\llbracket x_1\rrbracket)$ 
the matrix expression of the multiplication of $f$ on $A$ with respect to the $K\llbracket x_1\rrbracket$-basis $\Gamma$. 
Note that $M_f=f(M_{x_1},\dots,M_{x_r})$. 
As $A$ is a principal ideal domain, we have $\ell_{K\llbracket x_1\rrbracket}(\Coker M_f)=\ord_{x_1}(\det M_f)$ 
(see \cite{Fulton} Lemma A.2.6 for more general results). 
Since $\Coker M_f\cong A/fA$ as $K\llbracket x_1\rrbracket$-modules, we conclude that 
\[
\int(f;I)=\ord_{x_1}(\det M_f)=\ord_{x_1}(\det f(M_{x_1},\dots,M_{x_r})). 
\]
One can see how $\int(f-\alpha g;I)$ varies by computing $\det(M_f- a M_g)$ with a new variable $a$. 
In the special case where $I$ is generated by $h_1,\dots,h_n\in K[\x]_{\langle \x\rangle}$ such that $h_j$ mod $x_1$ is a monomial for all $j$, 
one can compute $M_{x_i}$ mod $x_1^N$ in finite time for any large $N>0$. 

We will see the case of a bivariate power series $F(x,y)\in K\llbracket x,y\rrbracket$ with $n=\ord_y F(0,y)$. 
By Weierstrass preparation theorem, we may assume that 
\[
F(x,y)=y^n+c_1(x)y^{n-1}+\dots+c_{n-1}(x)y+c_n(x)
\]
where $c_i(x)\in K\llbracket x\rrbracket$ with $c_i(0)=0$. 
Since $\langle F(x,y), x\rangle =\langle x,y^n\rangle$, 
$A=K\llbracket x,y \rrbracket/\langle F(x,y)\rangle$ is a free $K\llbracket x\rrbracket$-module with a basis $\{1,y,\dots,y^{n-1}\}$. 
As $M_x$ sends $y^i$ to $x\cdot y^{n-1}$ for $0\le i\le n-1$, $M_x=xE_n$ where $E_n$ is the unit matrix of rank $n$. 
Since $M_y$ sends $y^i$ to $y^{i+1}$ for $0\le i\le n-2$, and $y^{n-1}$ to $-(c_1(x)y^{n-1}+\dots+c_{n-1}(x)y+c_n(x))\in A$, 
$M_y$ is the companion matrix of $F$ as a monic polynomial in the variable $y$. 
\end{Observation}
\subsection{Main algorithm}
Let $I\subset K\llbracket \x\rrbracket$ be a one-dimensional radical ideal. 
We may assume $\bw(I)\in \N_+^r$ by pre-computation 
(one can decide whether $x_i\in I$ or not by comparing the initial ideals of $I$ and $I+\langle x_i\rangle$). 
Combining Algorithm \ref{Semi(p)} and the sufficient conditions for $I$ to be not prime (Theorem \ref{primeness} (1) and Theorem \ref{test 1}), 
we obtain an algorithm for deciding irreducibility of reduced algebroid curves. 
Concerning Theorem \ref{test 1}, we introduce an function testing the irreducibility of algebroid curves which we use in the main algorithm.  
\begin{Algorithm}\label{parametric_test}
Let $I$ be an unmixed ideal of dimension one, and $f,g \in K\llbracket \x\rrbracket$ with $\int(f;I)=\int(g;I)$. 
We define the function {\bf Parametric\_Test}$(f,g,I)$ whose output is of form $(``result", \mbox{Object})$ as follows: 
\begin{enumerate}
\item 
If $\int(f-\alpha g;I)<\int(f;I)$ for generic $\alpha\in K^\times$, 
then $``result"=``false"$ and 
$\mbox{Object}=\langle I,x_{r+1}-f,x_{r+2}-g\rangle)$. 
\item 
If $\int(f-\alpha g;I)=\int(f;I)$ for generic $\alpha\in K^\times$ and there exist $\beta_1,\beta_2\in K^\times$, $\beta_1\neq\beta_2$, 
such that $\int(f-\beta_ig;I)>\int(f;I)$ for $i=1,2$, then $``result"=``false"$ and 
$\mbox{Object}=\langle I,x_{r+1}-(f-\beta_1g), x_{r+1}-(f-\beta_2g)\rangle)$. 
\item 
If $\int(f-\alpha g;I)=\int(f;I)$ for generic $\alpha\in K^\times$ and there exist the unique $\beta\in K^\times$ 
satisfying $\int(f-\beta g;I)>\int(f;I)$, then $(``result", \mbox{Object})=(``not~~false", \beta)$  if $\int(f-\beta g;I)<\infty$, 
and $(``result", \mbox{Object})=(``false", \langle I,x_{r+1}-(f-\beta g)\rangle )$ otherwise.  
\end{enumerate}
\end{Algorithm}
\begin{Lemma}\label{radical test}
Let $I\subset K\llbracket \x\rrbracket$ be an unmixed ideal of dimension one, and $\w \in \T_{\rm loc}(I)$. 
Then $\T_{\rm loc}(I)=\pR\cdot \w$ and $\sqrt{\initial_\w(\p_1)}=\sqrt{\initial_\w(\p_2)}$ for any $\p_1,\p_2\in \Ass I$ 
if and only if $\sqrt{\initial_\w(I)}=\sqrt{\initial_\w(\p)}$ for any $\p\in \Ass I$. 
\end{Lemma}
\begin{proof}
One can prove this similarly to Theorem \ref{initial in ker}. 
\end{proof}
The following is the main algorithm in this paper. 
\begin{Algorithm}[Algorithm for deciding irreducibility of reduced algebroid curves]\label{main algorithm}
The following algorithm decides the primeness of a one-dimensional radical ideal $I\subset K\llbracket \x \rrbracket$. \\
{\bf Input:} a one-dimensional radical ideal $I\subset K\llbracket x_1,\dots,x_r \rrbracket$ with $\bw(I)\in \N_+^r$. \\
{\bf Output:} ($``result"$, $J$) 
such that $J$ is a one-dimensional ideal of $K\llbracket x_1,\dots,x_{r'} \rrbracket$, $r'\ge r$, and 
\begin{enumerate}
\item $K\llbracket x_1,\dots,x_r\rrbracket/I\cong K\llbracket x_1,\dots,x_{r'}\rrbracket/J$, 
\item if $I$ is prime, then $``result"=``true"$ and $\bw(J)$ is a tropism of $J$, 
\item if $I$ is not prime, then $``result"=``false"$ and $J$ has at least two tropisms. 
\end{enumerate}
\begin{enumerate}[1{:~}]
\item $\x$\subst $(x_1,\dots,x_r)$, $\w$\subst $\bw(I)=(w_1,\dots,w_r)$. 
\item \pscif{$\initial_\w(I)$ contains monomials}
\item \quad \pscreturn{($``false"$, $I$)}. 
\item \pscendif
\item \pscwhile{$\gcd(\w)\neq 1$}. 
\item \quad \pscif{$\LT_\prec(\sqrt{\initial_\w(I)})\neq \LT_\prec(\Prim(\w))$ for a term order $\prec$} 
\item \quad \quad $G$\subst a system of binomial generators of $\Prim(\w)$
\item \quad\quad \pscforall{$\x^\a-\x^\b\in G$ }
\item \quad\quad\quad $(``flag", \mbox{Object})$\subst  {\bf Parametric\_Test}$(\x^\a,\x^\b,I)$. 
\item \quad\quad\quad \pscif{$``flag"=``false"$}
\item \quad\quad\quad\quad \pscreturn ($``false", \mbox{Object}$). 
\item \quad\quad\quad\pscendif
\item \quad\quad\pscendforall
\item \quad\pscendif
\item \quad Take $f=\x^\a-\alpha\x^\b\in \sqrt{\initial_\w(I)} \backslash \initial_\w(I)$, $N$\subst ${\int(f;I)}$. 
\item \quad \pscif{$N=\infty$}
\item \quad\quad \pscreturn{($``false"$, $\langle I, x_{r+1}-f\rangle$)}. 
\item \quad \pscendif
\item \quad \pscwhile{$N\in \Semi(\w)$}
\item \quad\quad Take $\c\in \N^r$ such that $N=\c\cdot \w$.  
\item \quad\quad $(``flag", \mbox{Object})$\subst  {\bf Parametric\_Test}$(f,\x^\c,I)$. 
\item \quad\quad \pscif{$``flag"=``false"$}
\item \quad\quad\quad \pscreturn ($``false", \mbox{Object}$). 
\item \quad\quad\pscendif
\item \quad\quad $\beta$ \subst \mbox{Object},  $N$\subst ${\int(f-\beta\x^\c;I)}$, $f$\subst $f-\beta\x^\c$. 
\item \quad \pscendwhile
\item \quad $I$\subst $\langle I, x_{r+1}-f\rangle$, $\x$\subst $(\x,x_{r+1})$, $\w$\subst $(\w,N)$, $r$\subst $r+1$. 
\item \quad \pscif{$\initial_\w(I)$ contains monomials}
\item \quad\quad \pscreturn{($``false"$, $I$)}. 
\item \quad \pscendif
\item \pscendwhile
\item \pscreturn ($``true"$, $I$). 
\end{enumerate}
\end{Algorithm}
\begin{proof}
We will give a proof of the correctness of this algorithm. 

From Line 6 to 14: 
Note that $\w=\bw(I)\in \T_{\rm loc}(I)$ and $\gcd(\w)\neq 1$ at Line 6. 
Thus there exists $\p\in \Ass I$ such that $\T_{\rm loc}(\p)=\pR\cdot \w$. 
Note that $\sqrt{\initial_\w(I)})\subset \sqrt{\initial_\w(\p)}$. 
Suppose that $\LT_\prec(\sqrt{\initial_\w(I)})\neq\LT_\prec(\Prim(\w))$ is satisfied. 
Then $\sqrt{\initial_\w(I)}\neq \sqrt{\initial_\w(\p)}$ since $\LT_\prec(\sqrt{\initial_\w(\p)})=\LT_\prec(\Prim(\w))$ by Theorem \ref{initial in ker}. 
Hence $\T_{\rm loc}(I)\neq \pR\cdot \w$, 
or $\T_{\rm loc}(I)= \pR\cdot \w$ and $\sqrt{\initial_\w(\p_1)}\neq \sqrt{\initial_\w(\p_2)}$ for some $\p_1, \p_2 \in \Ass I$ by Lemma \ref{radical test}. 
In both cases, there exist $\x^\a-\x^\b\in \Prim(\w)$, $\alpha\in K$, and $\p'\in \Ass I$, such that 
$\x^\a-\alpha \x^\b\in \sqrt{\initial_\w(\p)}$ and $\x^\a-\alpha \x^\b \not\in \sqrt{\initial_\w(\p')}$. 
Therefore one can conclude that $I$ is not prime by testing {\bf Parametric\_Test}$(\x^\a,\x^\b, I)$ for all $\x^\a-\x^\b\in G$. 
Hence at Line 15, $\sqrt{\initial_\w(I)}=\sqrt{\initial_\w(\p)}$ for all $\p\in \Ass I$, 
and $\sqrt{\initial_\w(I)}\neq \initial_\w(I)$ since $\gcd(\w)\neq 1$. 

At Line 32, i$\w=\bw(I) \in \T_{\rm loc}(I)$ and $\gcd(\w)=1$, and thus $\bw(I)$ is a tropism of $I$. 

The output of this algorithm satisfies the desired conditions by Theorem \ref{primeness} and Theorem \ref{test 1}. 
To complete the proof, it is enough to show that this algorithm terminates in finite time. 

We will prove that the while loop from Line 19 to 26 terminates in finite time. 
If $I$ is prime, then this while loop is the same as the while loop from Line 4 to 8 in Algorithm \ref{Semi(p)}, 
and thus this while loop terminates in finite time. 
Suppose that $I$ is not prime, and assume, to the contrary, that this while loop causes an infinite loop. 
This means that there exist $f\in \sqrt{\initial_\w(I)} \backslash \initial_\w(I)$, 
$\{\beta_i\in K^\times\mid i\in \N\}$ and $\{\c_i\in \N^r\mid i\in \N\}$ satisfying following: 
Let $f_0=f$ and $f_{i+1}=f_i-\beta_i\x^{\c_i}$ for $i\in \N$. Then $\int(f_i;I)=\w\cdot \c_i$, $\int(f_i-\alpha\x^{\c_i};I)=\int(f_i;I)$ for generic $\alpha\in K^\times$, 
and $\beta_i$ is the unique $\beta\in K^\times$ satisfying $\int(f_i-\beta\x^{\c_i};I)>\int(f_i;I)$. 
Since $(\w\cdot\c_i)_{i\in\N}$ is a strictly increasing sequence and $\w\cdot \c_0=\int(f;I)>\ord_\w(f)$, $F:=\lim_{i\to\infty} f_i$ exists and $\initial_\w(F)=f$. 
Let $I=\p_1\cap\dots\cap\p_l$ be the prime decomposition of $I$, 
and define $A_j=K\llbracket\x\rrbracket/\p_j$ and $\phi_j:K\llbracket \x\rrbracket\to K\llbracket t\rrbracket$ as in the proof of Lemma \ref{test 1}. 
As we observed in the proof of Lemma \ref{test 1}, 
the uniqueness of $\beta_i$ shows that $\phi_j(f_i)=\beta_i \cdot\phi_j(\x^{c_i})$ for all $i\in \N$ and $1\le j\le l$. 
Hence $\lim_{i\to\infty} \ord_t(\phi_j(f_i))=\infty$ for all $j$. 
This show that $\phi_j(F)=0$ and thus $F\in \p_j$ for all $j$. 
Hence we conclude $F\in I$ which contradicts to $\initial_\w(F)=f\not\in \initial_\w(I)$. 

When the while loop from Line 5 to 31 completes, the semigroup $\Semi(\w)$ become larger. 
Hence this while loop terminates in finite time by Lemma \ref{acc for semigroup}. 

Therefore we conclude that this algorithm terminates in finite time. 
\end{proof}
\begin{Remark}
If one apply Algorithm \ref{main algorithm} to a non-radical ideal, it takes infinite time to complete the while loop from Line 19 to 26 in general. 
Note that the question of determining whether or not a given ideal is a radical ideal is uncomputable in general (\cite{Teitelbaum} Lemma 1). 
When we consider the implementation in computer algebra systems, we assume that 
$k$ is a computable perfect field (e.g. $\Q$ or $\F_p$), and $I=I_0 K\llbracket \x\rrbracket$ for some  $I_0\subset k[\x]$. 
Since $K[\x]$ is an excellent ring, one can compute $\sqrt{I}=\sqrt{I_0} K\llbracket \x\rrbracket$ in this situation. 
If $\alpha,\beta \in K$ are conjugate over $k$ and $f,g\in k[\x]$, then $\int(f-\alpha g;I)=\int(f-\beta g;I)$. 
Hence if the element $\alpha\in K$ satisfying $\int(f-\alpha g;I)>\int(f;I)$ is unique, 
then $\alpha\in k$. 
Thus all power series appearing in Algorithm \ref{main algorithm} are polynomials over $k$ except for the output. 
\end{Remark}
\begin{Example}[Kuo's Example revisited]\label{Kuo's Example revisited}
Let $F(x,y)=(y^2-x^3)^2-x^7\in K\llbracket x,y\rrbracket$, $\ch(k)\neq 2$. 
Then $\w:=\bw(F)=(4,6)$ and $\initial_\w(F)=(y^2-x^3)^2$. 
Let $g=y^2-x^3$. Then $\int(g;F)=14=(2,1)\cdot \bw(F)=\int(x^2y;F)$. 
We compute $\int(g-\alpha x^2y; F)$ as in Observation \ref{matrix method}. 
The algebroid curve $A=K\llbracket x,y \rrbracket/\langle F(x,y)\rangle$ is a free $K\llbracket x\rrbracket$-module with a basis $\{1,y,y^2,y^3\}$ 
and 
\[
M_x=
\left(
		\begin{array}{cccc}
		x & 0  & 0 & 0  \\
		0 & x  & 0 & 0  \\
		0 & 0  & x & 0  \\
		0 & 0  & 0 & x 
		\end{array}
\right),~~
M_y=
\left(
		\begin{array}{cccc}
		0 & 0  & 0 & -x^6+x^7  \\
		1 & 0  & 0 & 0        \\
		0 & 1  & 0 & 2x^3     \\
		0 & 0  & 1 & 0
		\end{array}
\right).
\]
Since 
\[
\int(g-\alpha x^2y; F)=\ord_x(\det(M_y^2-M_x^3-\alpha M_x^2M_y))=\ord_x((\alpha+1)^2(\alpha-1)^2x^{14}-\alpha^4x^{15}), 
\]
we have $\int(g-\alpha x^2y; F)=15$ if $\alpha= \pm 1$, and $\int(g-\alpha x^2y; F)=14$ otherwise. 
Thus $F$ is reducible, and $\langle F, z_1-(g-x^2y),z_2-(g+x^2y)\rangle \subset K\llbracket x,y,z_1,z_2\rrbracket$ has two tropisms $(2,3,7,8)$ and $(2,3,8,7)$. 
\end{Example}
\begin{Example}\label{space curve}
Let $I=\langle x^3-y^2, (z^2-xy)^2-x^2yz^2\rangle\subset K\llbracket x,y,z\rrbracket$ and set $A=K\llbracket x,y,z\rrbracket/I$. 
Then $\bw(I)=(8,12,10)$, $\initial_{\bw(I)}(I)=\langle x^3-y^2, (z^2-xy)^2\rangle$, and 
$\sqrt{\initial_{\bw(I)}(I)}=\langle x^3-y^2, z^2-xy\rangle=\Prim(8,12,10)$. 
Let $g=z^2-xy$. Then $\int(g;I)=2\cdot 4+6+2\cdot 5=24 =\int(y^2; I)$. 
We compute $\int(g-\alpha y^2;I)$ for $\alpha \in K$ as in Observation \ref{matrix method}. 

As $I+\langle x\rangle=\langle x,y^2,z^4\rangle$, $A$ is a free $K\llbracket x\rrbracket$-module with a basis $\{1,z,z^2,z^3,y,yz,yz^2,yz^3\}$. 
Since $y^3=x^2$ and $z^4=(2x+x^2) yz^2-x^5$ in $A$, one can compute $M_x$, $M_y$ and $M_z$ easily (for example, 
$M_z$ sends $yz^3\in A$ to $yz^4=(2x+x^2)y^2z^2-x^5y=(2x+x^2)x^3\cdot z^2-x^5\cdot y\in A$); 
\[
M_x=xE_8, ~~
M_y=
\left(
		\begin{array}{cc}
		O & x^3 E_4    \\
		E_4 & O \\
		\end{array}
\right),~~
M_z=
\left(
		\begin{array}{cccccccc}
		0 & 0  & 0 & -x^5     & 0 & 0  & 0 & 0   \\
		1 & 0  & 0 & 0         & 0 & 0  & 0 & 0   \\
		0 & 1  & 0 & 0         & 0 & 0  & 0 & (2x+x^2)x^3   \\
		0 & 0  & 1 & 0         & 0 & 0  & 0 &  0  \\
		0 & 0  & 0 & 0         & 0 & 0  & 0 & -x^5\\
		0 & 0  & 0 & 0         & 1 & 0  & 0 & 0   \\
		0 & 0  & 0 & 2x+x^2 & 0 & 1  & 0 & 0   \\
		0 & 0  & 0 & 0         & 0 & 0  & 1 & 0   \\
		\end{array}
\right).
\]
Since 
\begin{eqnarray*}
\int(g-\alpha y^2; I) &=& \ord_x (\det(M_z^2-M_xM_y-\alpha M_y^2))\\ 
&=& \ord_x((\alpha-1)^4(\alpha+1)^4 x^{24}+-2\alpha^2(\alpha+1)^2(\alpha-1)^2x^{25}+{\alpha}^4x^{26}), 
\end{eqnarray*}
we have $\int(g-\alpha x^2y; F)=26$ if $\alpha= \pm 1$, and $\int(g-\alpha x^2y; F)=24$ otherwise. 
Therefore $I$ is not prime, and $\langle I, u_1-(g-y^2), u_2-(g+y^2)\rangle \subset K\llbracket x,y,z, u_1,u_2\rrbracket$ has two tropisms 
$(4,6,5,14,12)$ and $(4,6,5,12,14)$. 
\end{Example}
\noindent{\bf Acknowledgments: }
This research was supported by JST, CREST. 
The author wishes to thank Thomas Markwig and Anders Nedergaard Jensen for their useful comments and discussions. 

\end{document}